\documentclass[11pt,reqno]{amsart}
\usepackage[T1]{fontenc}
\usepackage{graphicx}
\usepackage{setspace}
\usepackage{caption}
\usepackage{empheq}

\usepackage{color}
\definecolor{MyLinkColor}{rgb}{0,0,0.4}

\newtheorem{thm}{Theorem}[section]

\newtheorem{lem}[thm]{Lemma}

\theoremstyle{remark} 
\newtheorem{rem}[thm]{Remark}
\newtheorem{obs}[thm]{Observation}

\numberwithin{equation}{section}   

\title[Water flows with discontinuous vorticity and stagnation points]{Gravity water flows with discontinuous vorticity and stagnation points}

\author[C. I. Martin]{Calin Iulian Martin}
\address{Institut  f\" ur Mathematik, Universit\" at Wien, Oskar-Morgenstern-Platz 1,
1090 Wien, Austria}
\email{calin.martin@univie.ac.at}

\author[B.--V. Matioc]{Bogdan--Vasile Matioc}
\address{Institut f\" ur Angewandte Mathematik, Leibniz Universit\" at Hannover, Deutschland}
\email{matioc@ifam.uni-hannover.de}


\usepackage[colorlinks=true,linkcolor=MyLinkColor,citecolor=MyLinkColor]{hyperref} 

\begin{document}

\begin{abstract}
We construct small-amplitude steady periodic gravity water waves arising as the free surface of water flows that  contain stagnation points and possess 
a discontinuous distribution of vorticity in the sense that the flows consist  of two layers of constant  but different vorticities.
We also describe the streamline pattern in the moving frame for the constructed flows.
\end{abstract}

\maketitle

\section{Introduction}\label{Sec1}
We present here a study of steady periodic traveling water waves that propagate at the free surface of a two-dimensional inviscid and incompressible fluid of finite depth allowing for stagnation points and
for a discontinuous distribution of vorticity. 
More precisely, we consider water waves interacting with two vertically superposed currents of different constant vorticities.

Confined first to the investigation of waves of small amplitude, 
which can be satisfactorily approximated by sinusoidal curves within the linear theory, the examination of periodic traveling water waves arising as the free surface 
of an irrotational flow with a flat bed originates at the beginning of the $19^{th}$ century. The description of waves that are flatter near the trough and have 
steeper elevations near the crest necessitates a nonlinear approach, which was in fact conducted in the last decades and led to the 
first rigorous results concerning the existence of wave trains in irrotational flow, see for instance the case of   Stokes waves \cite{To96} and the flow beneath them (particle trajectories, behavior of the pressure) 
cf. \cite{CoCla13, Co06, CoARMA13, CoSt10}.

To go beyond irrotational flows and to treat wave current interactions one needs to incorporate vorticity into the problem, cf. \cite{Cobook, Jon, Thom}.
However, the difficulties generated by the presence of the vorticity have prevented a rigorous mathematical development,
which appeared only relatively recently in \cite{CoSt04}, where the existence of small and large amplitude steady periodic gravity water waves with a general (continuous) vorticity distribution was proved. 

Of high significance is the investigation of steady periodic rotational waves interacting with currents that possess a  rough  -- that is discontinuous or unbounded -- vorticity.
Discontinuous vorticities model sudden changes in the underlying current, numerical simulations of such flows being quite recent \cite{KO1,KO2}.
Unbounded vorticities on the other hand can describe turbulent flows in channels (see the empirical law in on page 106 of \cite{B62}) and are relevant also in the setting of 
wind generated waves that possess a thin layer of high vorticity adjacent to the wave surface \cite{O82, PB74}.
The discontinuous vorticity distribution was considered in the 
groundbreaking paper \cite{CoSt11} where  the existence of steady two-dimensional periodic gravity
water waves of small and large
amplitude on water  flows with an arbitrary bounded (but discontinuous) vorticity was proved.
Small amplitude capillary-gravity waves with discontinuous but bounded vorticity were constructed in \cite{CM14, M13xx}.
Waves with unbounded vorticity were first shown to exist in \cite{CM13xxx} but only when allowing for  surface tension as a restoring force.
This situation appears in 
many physical settings one of which being that of wind blowing over a still 
fluid surface and giving rise to two-dimensional small amplitude wave trains driven by capillarity \cite{Ki65} which grow larger and turn into capillary-gravity waves.

Another striking occurrence in water flows is the presence of stagnation points, that is points where the steady velocity field vanishes, thus making the analysis more intricate, since the 
usual Dubreil-Jacotin transform which converts the original free boundary problem into a problem in a fixed domain, is no longer available.
There is a short list of papers dealing
with existence of water flows allowing for stagnation points and for a  non-vanishing continuous vorticity, cf. \cite{CV11, EW14x, EMM11, KK14, W09} for gravity waves and \cite{CM12x, CM13, CM13_2, M13x} for waves with capillarity. 
Under consideration in this paper is a a more involved setting where, in addition to permitting stagnation points (whose existence in the fluid is proven), we also allow for a discontinuous  distribution of the vorticity. 
To our best knowledge the incorporation of both stagnation points and of a discontinuous vorticity is a feature that was not rigorously
analyzed before.

The governing equations are the Euler equations of motion, together with boundary conditions on the free surface and on the flat bed of the water flow. The discontinuous vorticity that
we consider here is of the following type: we assume that the flow has a layer of constant vorticity $\gamma_2$ adjacent to the free surface above another layer of constant vorticity
$\gamma_1$ neighboring the flat bed. Of course, the interesting situation (that we pursue here) is when $\gamma_1\neq \gamma_2$. 
The unknowns are here the free surface, the interface
separating the regions of different vorticities (which can be seen as an internal wave due to the discontinuity in vorticity), the velocity field and the pressure function. 
In a first step we reduce the number of unknowns by means of the stream function whose 
utilization converts the problem into a transmission problem along the line of discontinuity of vorticity with fewer unknowns. The second step that we undertake is to consider a
flattening transformation which has the advantage that changes the free boundary value problem into a problem in a fixed domain, thus making it more tractable for the analysis. For studying
the latter resulted problem we employ the Crandall-Rabinowitz Theorem on bifurcation from simple eigenvalues.

The dispersion relation that we obtain -- which is a formula giving the speed at the free surface of the bifurcation inducing laminar flows in terms of the two vorticities $\gamma_1, \gamma_2$, the thickness of 
the two rotational layers and the wavelength -- generalizes the one in \cite{CV11} obtained in the case of a water flow with constant vorticity and allowing for stagnation points.
The intricacy of the dispersion relation -- a third order algebraic equation -- allows us to prove existence of water waves of small wavelength arising as the free surface of 
water flows with rotational layers of different constant vorticities and containing stagnation points, cf. Theorems \ref{MT1}-\ref{MT3}, \ref{MT4}, \ref{MT5}. 
We present also the streamline pattern in the moving frame for the solutions, cf. Figures \ref{Fig1}-\ref{Fig3}.
Our results show especially that the ratio of the amplitudes of the   surface wave and that of the internal wave -- and the fact that the surface wave and the internal wave are in phase or anti phase -- is highly 
influenced by the vorticities  of the currents and by the speed at the free surface of the bifurcation inducing laminar flows.

We briefly outline the content of the paper. We present in Section \ref{Sec2} the governing equations
together with the analytic setting we work in. Moreover, we also find the dispersion relation whose
analysis is undertaken in Section \ref{Sec:3} for the
case $\gamma_2>0$, while the more singular case 
$\gamma_2=0$ is treated in Section \ref{Sec:4}. 
The Appendix  contains several technical lemmas.

\section{The governing equations}\label{Sec2}
Under consideration is a two-dimensional steady periodic flow, moving under the influence of  gravity, such that the  surface waves propagate in the positive $x$-direction.
The water flow occupies the domain $\Omega$ bounded below by the flat bed $y=-d,$ with $d>0$, and above by the free surface $y=h(x)$, which is a small perturbation of the flat free surface $y=0$.
In a reference frame moving with the wave speed $c>0$, the equations of motion in $\Omega$ are Euler's equations
 \begin{subequations}\label{VF}
 \begin{equation}\label{Euler}
\left\{
\begin{array}{rllll}
({u}-c) { u}_x+{ v}{ u}_y&=&-{ P}_x,\\
({ u}-c) { v}_x+{ v}{ v}_y&=&-{ P}_y-g,\\                                 
{ u}_x+{v}_y&=&0,
\end{array}\right.
\end{equation}
where $(u,v)$ denotes the velocity field, $P$ stands for pressure and $g$ is the    gravity constant.
The equations of motion are supplemented by the boundary conditions, which, ignoring surface tension effects, are
\begin{equation}\label{EQBC}
\left\{
\begin{array}{rllll}
 P&=&{P}_0&\text{on $ y=h(x)$},\\
 v&=&({ u}-c)h'&\text{on $ y=h(x)$},\\
 v&=&0 &\text{on $ y=-d$},
\end{array}
\right.
\end{equation}
with $P_0$ being the constant atmospheric pressure.

We are interested in solutions of the problem \eqref{VF} for which the vorticity  $\omega:=u_y-v_x$ of the flow
presents discontinuities of the following type: we assume that, adjacent to the free surface, the water flow possesses a layer $$\Omega(f,h):=\{(x,y):x\in\mathbb{R}, -d_2+f(x)<y<h(x)\},$$
of constant   vorticity $\gamma_{2}$, situated
above another layer $$\Omega(f):=\{(x,y):x\in
\mathbb{R}, -d<y<-d_2+f(x)\},$$ 
which is adjacent to the flat bed and is of constant  vorticity $\gamma_{1}$, that is
\begin{equation}\label{vorticity}
\omega:=
\left\{
\begin{array}{lll}
 \gamma_1,& \text{in $\Omega(f)$},\\
 \gamma_2,&\text{in $\Omega(f,h).$}
\end{array}
\right.
\end{equation} 
\end{subequations}

We steadily assume  that $\gamma_1\neq \gamma_2$ and that $d_2>0$,  $d-d_2=:d_1>0.$ 
We note that, additionally to  $(u,v,P,h),$ we have a further unknown: the function $f$ whose graph separates the two currents of different vorticities.
By Helmholtz's law, the vorticity is constant along streamlines of the steady flow, and as a consequence of this $y=-d_2+f$ has to be a streamline of the flow.
This streamline can be viewed as an internal wave  due to the jump in vorticity. 

With the help of the stream function $\psi$, introduced (up to an additive constant) via the relation $\nabla\psi=(-v,u-c)$ we can reformulate \eqref{Euler}-\eqref{vorticity} as the 
free-boundary problem
\begin{subequations}\label{PB2}
\begin{equation}\label{tra_eq}
 \left\{\begin{array}{lll}
         \Delta\psi_2 =\gamma_2 & \text{in $\Omega(f,h),$}\\
         \Delta\psi_1 =\gamma_1 & \text{in $\Omega(f),$}\\
         \psi_2 = 0 &\text{on $y=h(x)$,}\\
         \psi_2 =\psi_1 &\text{on $ y=-d_2+f(x),$}\\
         \psi_1= m &\text{on $ y=-d,$}
        \end{array}\right.
\end{equation}
subjected to the conditions
\begin{equation}\label{trans}
\left\{\begin{array}{lll}
        \partial_y \psi_2 =\partial_y\psi_1 & \text{on $ y=-d_2+f(x),$}\\[1ex]
         \vert\nabla \psi_2\vert^2+2g(d+h)= Q & \text{on $ y=h(x),$}
\end{array}\right.
\end{equation} 
\end{subequations} 
where  the constant $-m$ represents the relative mass flux and $Q\in\mathbb{R}$ is related with the hydraulic head.
Moreover, $\psi_1:=\psi\big|_{\Omega(f)}$ and $\psi_2:=\psi\big|_{\Omega(f,h)},$ so that from the fourth equation of \eqref{tra_eq} and the first equation of \eqref{trans} we
see that the $\nabla \psi$ (hence also the velocity field) is continuous across the interface $y=-d_2+f(x)$.

Given $\alpha\in(0,1),$ it is not difficult to see that any solution 
\[
((f,h),\psi_1,\psi_2)\in \big(C^{3+\alpha}_{per}(\mathbb{R})\big)^2\times C^{3+\alpha}_{per}\big(\, \overline{\Omega(f)}\, \big)\times C^{3+\alpha}_{per}\big(\, \overline{\Omega(f,h)}\, \big) 
\]
of \eqref{PB2} defines a solution 
\begin{align*}
&(u,v,P,(f,h))\in \big(C_{per}^{1-}(\overline\Omega)\big)^3\times \big(C^{3+\alpha}_{per}(\mathbb{R})\big)^2\\[-1ex]
&\big((u,v)\big|_{\Omega(f)},(u,v)\big|_{\Omega(f,h)}\big) \in \big(C^{2+\alpha}_{per}\big(\, \overline{\Omega(f)}\, \big)\big)^2\times \big(C^{2+\alpha}_{per}\big(\, \overline{\Omega(f,h)}\, \big)\big)^2\\[-1ex]
&\big(P\big|_{\Omega(f)},P\big|_{\Omega(f,h)}\big)\in C^{2+\alpha}_{per}\big(\, \overline{\Omega(f)}\, \big)\times C^{2+\alpha}_{per}\big(\, \overline{\Omega(f,h)}\, \big)
 \end{align*}
of \eqref{VF}.  
The subscript  {\em per}  stands for functions that are periodic in the horizontal variable, meaning that all the functions considered above are $L-$periodic with respect to $x,$ with $L>0$ being fixed.

We first determine laminar flow solutions of problem \eqref{PB2}, that is water flows with a flat free surface and parallel streamlines, meaning that they present no $x$-dependence.
Of interest are laminar flows that contain stagnation points,  more precisely laminar flows that contain streamlines consisting entirely of stagnation points.
Then we study when non laminar solutions bifurcate from the laminar flows and  describe the qualitative picture of
the streamline pattern for the bifurcating solutions.

\paragraph{\bf Laminar flow solutions}
Because the stream function is constant along the streamline $y=-d_2+f(x)$, we use the value of the stream function
\begin{equation}\label{lambda}\psi_1=\psi_2=\lambda\qquad \text{on $ y=-d_2+f(x),$}\end{equation}
to parametrize a family of laminar solutions of \eqref{tra_eq}.
Setting $f\equiv h\equiv 0$ we obtain from \eqref{tra_eq} that the stream function $\psi^0:=(\psi_1^0,\psi_2^0)$  satisfies
\begin{align*}
& \psi^{0}_{1}(y)=\frac{\gamma_1 y^2}{2}+\Big(\frac{\gamma_1(d+d_2)}{2}+\frac{\lambda-m}{d_1}\Big)y+\frac{\lambda d}{d_1}+\frac{  \gamma_1dd_2}{2}-\frac{md_2}{d_1},\quad y\in [-d,-d_2]\\[-1ex]
& \psi^{0}_{2}(y)=\frac{\gamma_2 y^2}{2}+\Big(\frac{\gamma_2 d_2}{2}-\frac{\lambda}{d_2}\Big)y,\quad y\in [-d_2,0].
\end{align*}
The  equations of \eqref{trans} are equivalent to 
\begin{equation}\label{FK}
 m=\frac{\lambda d}{d_2}+d_1\frac{\gamma_1d_1+\gamma_2 d_2}{2},\qquad
 Q=\Big(\frac{\gamma_2 d_2}{2}-\frac{\lambda}{d_2}\Big)^2+2gd.
\end{equation}
In the following we choose the constants $m$ and $Q$ in \eqref{tra_eq} and \eqref{trans} to be given by \eqref{FK}, the constant $\lambda$ introduced via \eqref{lambda}  being left as a  parameter.
Hence,   each $\lambda\in\mathbb{R}$ determines a unique laminar solution $((f,h),\psi_1,\psi_2):=(0,\psi_1^0,\psi_2^0)$  of \eqref{PB2} when $m$ and $Q$ are defined by \eqref{FK}.
These are the laminar solutions from which we study bifurcation. 

\paragraph{\bf Conditions for stagnation}
We note that the laminar flows determined above  possess stagnation points -- that is water particles that travel horizontally with the wave speed -- if and only if 
\begin{equation}\label{prelstag}
 \partial_y \psi^0_2 (-d_2)\cdot\partial_y\psi^0_2(0)\leq0\qquad\text{or}\qquad \partial_y \psi^0_1 (-d)\cdot\partial_y\psi^0_1(-d_2)\leq 0.
\end{equation}
If {\eqref{prelstag} holds true, then there exists $y_0\in [-d,0]$ such that 
$$\partial_y \psi^0_i(y_0)=0\quad{\rm for}\quad i=1\,\,{\rm or}\,\,2.$$ 
The streamline $y=y_0$} consists only of stagnation points, and we expect that  the solutions of \eqref{VF}
that bifurcate from these laminar solutions possess stagnation points too, cf. \cite{EEW11, W09}.
The first inequality ensures  stagnation in the layer adjacent to the wave surface, and is equivalent to
\begin{equation}\label{top}
  \Lambda (\Lambda- \gamma_2d_2)\leq0,
\end{equation}
respectively the condition for stagnation in the bottom layer is 
\begin{equation}\label{bed}
 (\Lambda-\gamma_2d_2)( \Lambda -\gamma_1d_1-\gamma_2d_2)\leq0.
\end{equation}
Hereby, we set
\begin{equation}\label{Lam}
\Lambda:=\frac{\gamma_2  d_2}{2}-\frac{\lambda}{d_2}. 
\end{equation}
The constant $\Lambda$ has a physical interpretation: it is the relative horizontal speed at the free surface for the laminar flow determined by $\lambda,$ that is $\Lambda=\partial_y\psi_2^0\big|_{y=0}.$
For this reason   we define $\lambda$ via \eqref{Lam} and use $\Lambda$ as   parameter.

\paragraph{\bf The analytic setting} With $\Lambda$ as parameter, we are left to seek special values of $\Lambda$ such that branches of non laminar solutions of \eqref{PB2} bifurcate  from the curve of laminar flows.
For this, we need to recast \eqref{PB2} in a suitable analytic setting.

In the following $\alpha\in(0,1)$ is a fixed H\"older exponent.
Because the equations of \eqref{tra_eq} and \eqref{trans} are posed on manifolds that depend on the unknown functions $(f,h)$, it is suitable to transform the problem \eqref{PB2} on fixed manifolds. 
For this, we set $\Omega_1:=\Omega(0)$, $\Omega_2:=\Omega(0,0)$ and define the mappings
\begin{align*}
&\Phi_f:\Omega_1\rightarrow\Omega(f),\quad \Phi_f (x,y)=\Big(x,\frac{d_1+f(x)}{d_1}y+\frac{d}{d_1}f(x)\Big),\\
 &\Phi_{(f,h)}:\Omega_2\rightarrow\Omega(f,h),\quad \Phi_{(f,h)}(x,y)=\Big( x,\frac{h(x)-f(x)+d_2}{d_2}y+h(x)\Big).
\end{align*}
 It is easy to  see that $\Phi_f$ and $\Phi_{(f,h)}$ are $C^{3+\alpha}-$diffeormorphisms for each $(f,h)\in\mathcal{O}$, whereby
 \[ \mathcal{O}:=\{(f,h)\in\big(C^{3+\alpha}_{e,per}(\mathbb{R})\big)^2\,:\, -d<-d_2+f<h\},\]
 the subscript $e$ referring to the fact that we consider only even function in $x$.
Using these diffeomorphisms, we define the linear elliptic operators 
\begin{align*}
& \mathcal{A}(f):C^{3+\alpha}_{e,per}(\overline{\Omega}_1)\rightarrow C^{1+\alpha}_{e,per}(\overline{\Omega}_1),\quad \mathcal{A}(f)w_1:=\Delta (w_1\circ\Phi_f^{-1})\circ\Phi_f,\\[1ex]
 &\mathcal{A}(f,h):C^{3+\alpha}_{e,per}(\overline{\Omega}_2)\rightarrow C^{1+\alpha}_{e,per}(\overline{\Omega}_2),\quad \mathcal{A}(f,h)w_2:=\Delta (w_2\circ\Phi_{(f,h)}^{-1})\circ\Phi_{(f,h)},
\end{align*}
and the boundary operators
\begin{align*}
 &\mathcal{B}_1:\mathbb{R}\times \mathcal{O}\times C^{3+\alpha}_{e,per}(\overline\Omega_2)\to C^{2+\alpha}_{e,per}(\mathbb{R}),\\
 &\mathcal{B}_2:\mathbb{R}\times \mathcal{O}\times C^{3+\alpha}_{e,per}(\overline\Omega_1)\times C^{3+\alpha}_{e,per}(\overline\Omega_2)\to C^{2+\alpha}_{e,per}(\mathbb{R}),
\end{align*}
respectively through
\begin{align*}
&\mathcal{B}_1(\Lambda,(f,h),w_2):=\Big(|\nabla (w_2\circ\Phi_{(f,h)}^{-1})|^2\circ\Phi_{(f,h)}+2g(d+h)-Q\Big)\Big|_{y=0},\\[1ex]
&\mathcal{B}_2(\Lambda,(f,h))[w_1,w_2]:=\Big[\big(\partial_y(w_2\circ\Phi_{(f,h)}^{-1})\big)\circ\Phi_{(f,h)} -\big(\partial_y(w_1\circ\Phi_f^{-1})\big)\circ\Phi_f\Big]\Big|_{y=-d_2}.
\end{align*}

\begin{obs}\label{Obs:1} Let $\Lambda\in\mathbb{R}$, $((f,h),\psi_1,\psi_2)\in\mathcal{O}\times C^{3+\alpha}_{per}\big(\, \overline{\Omega(f)}\, \big)\times C^{3+\alpha}_{per}\big(\, \overline{\Omega(f,h)}\, \big)$, and assume that $\lambda, m, Q$ are defined by \eqref{lambda}, 
\eqref{FK}, and \eqref{Lam}.
Then,  the tuple $((f,h),\psi_1,\psi_2)$  solves the problem  \eqref{PB2}   if and only if   
\begin{itemize}
    \item[$(i)$] $w_1:=\psi_1\circ\Phi_f\in C^{3+\alpha}_{e,per}(\overline\Omega_1)$ is the unique solution of the Dirichlet problem
    \begin{equation}\label{DP1}
 \left\{\begin{array}{llll}
         \mathcal{A}(f)w_1  = \gamma_1 & \text{in $\Omega_1$},\\
         w_1 =\lambda&\text{on $ y=-d_2,$}\\
         w_1=m&\text{on $y=-d.$}
        \end{array}\right.
\end{equation}
\item[$(ii)$] $w_2:=\psi_2\circ \Phi_{(f,h)}\in C^{3+\alpha}_{e,per}(\overline\Omega_2)$ is the unique solution of  the Dirichlet problem
\begin{equation}\label{DP2}
 \left\{\begin{array}{llll}
         \mathcal{A}(f,h)w_2  = \gamma_2&\text{in $\Omega_2$},\\
         w_2 =0&\text{on $ y=0,$}\\
         w_2=\lambda&\text{on $y=-d_2.$}
        \end{array}\right.
\end{equation}
\item[$(iii)$] $\mathcal{B}_1(\Lambda,(f,h),w_2)=\mathcal{B}_2(\Lambda,(f,h))[w_1,w_2]=0$ in $C^{2+\alpha}_{e,per}(\mathbb{R}).$
   \end{itemize}
 \end{obs}
  Thanks to Observation \ref{Obs:1} we can recast the problem  \eqref{PB2} as a nonlinear and nonlocal equation with $(\Lambda,(f,h))$ as unknown.
  In order to proceed, we establish first the following result. 
  \begin{lem}\label{L:1} Given $(\Lambda,(f,h))\in\mathbb{R}\times\mathcal{O}$, we let $w_1:=w_1(\Lambda,(f,h))$ and $w_2:=w_2(\Lambda,(f,h))$ denote the unique solution of \eqref{DP1} and \eqref{DP2}, respectively, with $\lambda$ given by \eqref{Lam}. 
   Then, we have $w_i\in C^\omega\big(\mathbb{R}\times\mathcal{O}, C^{3+\alpha}_{e,per}(\overline\Omega_i)\big), i=1,2.$
  \end{lem}
\begin{proof}
 We prove just the real-analyticity of the solution operator  $w_1,$ the claim for $w_2$ following similarly. 
 By elliptic theory, cf. e.g. \cite{GT01}, we see that $w_1:\mathbb{R}\times\mathcal{O}\to C^{3+\alpha}_{e,per}(\overline\Omega_1)$ is well-defined.
 Moreover, we have that 
 \[\mathcal{F}(\Lambda, (f,h), w_1(\Lambda,(f,h)))=0\qquad\text{for all $(\Lambda,(f,h))\in\mathbb{R}\times\mathcal{O}$,}\]
 whereby $\mathcal{F}\in C^\omega\big(\mathbb{R}\times\mathcal{O}\times C^{3+\alpha}_{e,per}(\overline\Omega_1), C^{1+\alpha}_{e,per}(\overline\Omega_1)\times ( C^{3+\alpha}_{e,per}(\mathbb{R}))^2\big)$ is the operator defined by
 \[
 \mathcal{F}(\Lambda, (f,h), w_1):=(\mathcal{A}(f)w_1  - \gamma_1, w_1\big|_{y=-d_2}, w_1\big|_{y=-d}).
 \]
 Taking into account that Fr\'echet derivative
 \[
 \partial_{w_1}\mathcal{F}(\Lambda, (f,h), w_1((\Lambda,(f,h))))[z]:=(\mathcal{A}(f)z, z\big|_{y=-d_2}, z\big|_{y=-d})
 \]
 is an isomorphism, the assertion follows from the implicit function theorem. 
\end{proof}

Because $\mathcal{B}_i, i=1,2,$ depend real-analytically on their arguments too, we obtain from Lemma \ref{L:1} and Observation \ref{Obs:1} that the problem \eqref{PB2} is equivalent to the nonlinear and nonlocal equation
\begin{equation}\label{NP}
 \Phi (\Lambda, (f,h))=0,
\end{equation}
\smallskip
whereby $\Phi:=(\Phi_1,\Phi_2)\in C^\omega\big(\mathbb{R}\times\mathcal{O}, \big(C^{2+\alpha}_{e,per}(\mathbb{R})\big)^2\big)$ is the operator defined by
\begin{equation}\label{Phi}
\Phi(\Lambda,(f,h)):=(\mathcal{B}_1(\Lambda,(f,h),w_2(\Lambda,(f,h))), \mathcal{B}_2(\Lambda,(f,h))[w_1(\Lambda,(f,h)),w_2(\Lambda,(f,h))]).
\end{equation}
The laminar flow solutions of \eqref{PB2} correspond to the trivial solutions $(\Lambda,0)\in\mathbb{R}\times \mathcal{O}$ of 
\eqref{NP}.
In order to find other solutions, we use the theorem on bifurcations from simple eigenvalues due to Crandall and Rabinowitz  \cite{CR71}.

\begin{thm}[Crandall and Rabinowitz]\label{CR}
Let $\mathbb{X},\mathbb{Y}$ be real Banach spaces   and  let the mapping $\Phi\in C^\omega(\mathbb{R} \times \mathbb{X},\mathbb{Y})$ satisfy:
\begin{enumerate}
\item[(a)] \label{cr1} $\Phi(\Lambda,0)=0$ for all $\Lambda\in \mathbb{R}$;
\item[(b)] \label{cr2} There exists $\Lambda_*\in \mathbb{R}$ such that Fr\'echet derivative $\partial_x \Phi(\Lambda_*,0)$ is 
a Fredholm operator of index zero with a one-dimensional kernel and 
\[{\rm Ker\,} \partial_x \Phi(\Lambda_*,0)={\rm span\, }\{x_0\} \qquad\text{with $ 0\neq x_0\in \mathbb{X}$};\]
\item[(c)] \label{cr3}
 The transversality condition 
\[
\partial_{\Lambda x}\Phi(\Lambda_*,0)[x_0]\not \in {\rm Im\,}\partial_x \Phi(\Lambda_*,0).
\]
\end{enumerate}
Then, $(\Lambda_*,0)$ is a bifurcation point in the sense that there exists $\varepsilon>0$ and a real-analytic
curve $(\Lambda,x):(-\varepsilon,\varepsilon)\to \mathbb{R}\times \mathbb{X}$  consisting only of solutions of the equation $\Phi(\Lambda,x)=0$. 
Moreover, as $s\to 0,$ we have that
\[
\Lambda(s)=\Lambda_*+O(s) \qquad\text{and}\qquad x(s)=sx_0 +O(s^2).
\]
Furthermore, there exists an open set $U\subset \mathbb{R}\times \mathbb{X}$ with $(\Lambda_*,0)\in U$ and 
\[
\{(\Lambda,x)\in U \,:\, \Phi(\Lambda,x)=0, x\neq 0\}=\{(\Lambda(s),x(s))\, : \, 0<|s|<\varepsilon\}.
\]
\end{thm}

In order to apply this abstract bifurcation result, we need to compute the Fr\'echet derivative of the operator $\Phi$.
To this end we state the following lemma.
\begin{lem}\label{L:FD} 
 Let $\Lambda\in\mathbb{R}$ be given.
 The Fr\'echet derivative $\partial_{(f,h)}\Phi(\Lambda,0)$ is the matrix operator
 \[
 \partial_{(f,h)}\Phi(\Lambda,0)=\begin{pmatrix}
                            A_{11}&A_{12}\\
                           A_{21}&A_{22} 
                           \end{pmatrix}\in\mathcal{L}\big(\big(C^{3+\alpha}_{e,per}(\mathbb{R})\big)^2, \big(C^{2+\alpha}_{e,per}(\mathbb{R})\big)^2\big).
 \]
 Given $1\leq i,j\leq 2,$ the operator $A_{ij}\in \mathcal{L}\big(C^{3+\alpha}_{e,per}(\mathbb{R}), C^{2+\alpha}_{e,per}(\mathbb{R})\big) $ is the   Fourier multiplier with symbol $m^{ij}(\Lambda):=(m^{ij}_k(\Lambda))_{k\in\mathbb{N}}$ defined  by
 \begin{align}
  &m^{11}_k(\Lambda)=-2 \Lambda(\gamma_2 d_2-\Lambda)\frac{R_k}{\sinh\left(R_kd_2\right)},\label{a11}\\
  &m^{12}_k(\Lambda)=2 \Big[g+\gamma_2\Lambda-\Lambda^2 \frac{R_k}{\tanh\left(R_kd_2\right)}\Big],\label{a12}\\
  &m^{21}_k(\Lambda)=\gamma_2-\gamma_1+(\Lambda-\gamma_2d_2) \Big[ \frac{R_k}{\tanh\left(R_kd_1\right)}+ \frac{R_k}{\tanh\left(R_kd_2\right)}\Big],\label{a21}\\
  &m^{22}_k(\Lambda)=-\Lambda\frac{R_k}{\sinh\left(R_kd_2\right)}\label{a22}
 \end{align}
for $k\in\mathbb{N}$, whereby $R_k:=2k\pi/L.$ 
For $k=0$ the right-hand side of \eqref{a11}-\eqref{a22} should be understood as the limit of the expressions  when letting $R_k\to 0.$
\end{lem}
\begin{proof}
 See Appendix.
\end{proof}

With the help of Lemma \ref{L:FD} we are now able to determine when the Fr\'echet derivative  $\partial_{(f,h)}\Phi(\Lambda,0)$ is a Fredholm operator.  
\begin{lem}\label{L:FP}
 Let $\Lambda\in\mathbb{R}$ be given.
 We have:
 \begin{itemize}
  \item[$(i)$] If $\Lambda\in\{0,\gamma_2d_2\},$ then $\partial_{(f,h)}\Phi(\Lambda,0)$ is  not a Fredholm operator.
    \item[$(ii)$] If $\Lambda\not\in\{0,\gamma_2d_2\},$ then $\partial_{(f,h)}\Phi(\Lambda,0)$ is  a Fredholm operator of index zero.
 \end{itemize}
\end{lem}
\begin{proof}
 In order to prove $(i)$, we infer from \eqref{a11} and \eqref{a12} that for $\Lambda=0$ we have 
 \[
 \partial_{(f,h)}\Phi_1(\Lambda,0)[(f,h)]=2gh \qquad\text{for all $(f,h)\in\big(C^{3+\alpha}_{e,per}(\mathbb{R})\big)^2$,}
 \]
 meaning that ${\rm Im\,} \partial_{(f,h)}\Phi_1(\Lambda,0)=C^{3+\alpha}_{e,per}(\mathbb{R}).$
 Since $C^{3+\alpha}_{e,per}(\mathbb{R})$ is 
 not a closed subspace of $C^{2+\alpha}_{e,per}(\mathbb{R})$, the assertion is evident.
 Furthermore, if $\Lambda=\gamma_2d_2\neq0,$ then 
 \[
 \partial_{(f,h)}\Phi_2(\Lambda,0)[(f,h)]=(\gamma_2-\gamma_1)f+\mathcal{K}h \qquad\text{for all $(f,h)\in\big(C^{3+\alpha}_{e,per}(\mathbb{R})\big)^2$,}
 \]
 whereby
 \[
\mathcal{ K}\sum_{k\in\mathbb{N}} b_k\cos(R_kx)=-\Lambda\sum_{k\in\mathbb{N}} \frac{R_k}{\sinh(R_kd_2)}b_k\cos(R_kx).
 \]
 It is not difficult to see that $\mathcal{K}\big(C^{3+\alpha}_{e,per}(\mathbb{R})\big)\subset C^{\infty}_{e,per}(\mathbb{R}),$ so that ${\rm Im\,} \partial_{(f,h)}\Phi_2(\Lambda,0)= C^{3+\alpha}_{e,per}(\mathbb{R}).$ 
 Hence, ${\rm Im\,} \partial_{(f,h)}\Phi(\Lambda,0)$ is not  a closed subspace of $\big(C^{2+\alpha}_{e,per}(\mathbb{R})\big)^2.$
 This proves $(i)$.
 
 To prove $(ii)$,  we choose $\Lambda\not\in\{0,\gamma_2d_2\}$ and set
 \[D(k,\Lambda):=m_k^{11}(\Lambda)m_k^{22}(\Lambda)-m_k^{12}(\Lambda)m_k^{21}(\Lambda),\qquad k\in\mathbb{N}.\]
 From \eqref{a11}-\eqref{a22} it is clear that there exists $k_0\in\mathbb{N}$ such that $D(k,\Lambda)\neq0$ for all $k\geq k_0.$
 Defining the symbols $\widetilde m^{ij}(\Lambda)$ by $\widetilde m^{ij}_k(\Lambda)= m^{ij}_k(\Lambda)$ for $k\geq k_0$ and $1\leq i,j\leq 2$ and 
 \[\text{$\widetilde m^{11}_k(\Lambda)=\widetilde m^{22}_k(\Lambda)=1$,\quad $\widetilde m^{12}_k(\Lambda)=\widetilde m^{21}_k(\Lambda)=0$ \qquad for $0\leq k\leq k_0-1,$}\]
 we see that $\partial_{(f,h)}\Phi(\Lambda,0)$ is a compact perturbation of the operator 
 \[
 T:=\begin{pmatrix}
                            \widetilde A_{11}&\widetilde A_{12}\\
                           \widetilde A_{21}&\widetilde A_{22} 
                           \end{pmatrix}\in\mathcal{L}\big(\big(C^{3+\alpha}_{e,per}(\mathbb{R})\big)^2, \big(C^{2+\alpha}_{e,per}(\mathbb{R})\big)^2\big),
 \]
 where  $\widetilde A_{ij}\in \mathcal{L}\big(C^{3+\alpha}_{e,per}(\mathbb{R}), C^{2+\alpha}_{e,per}(\mathbb{R})\big) $ is the   Fourier multiplier defined by $\widetilde m^{ij}(\Lambda)$, $1\leq i,j\leq 2.$
 Because  $\widetilde D(k,\Lambda):=\widetilde m_k^{11}(\Lambda)\widetilde m_k^{22}(\Lambda)-\widetilde m_k^{12}(\Lambda)\widetilde m_k^{21}(\Lambda)\neq0 $ for all $k\in\mathbb{N},$ we can define the  formal inverse of  $T$ by  
 \[
 S:=\begin{pmatrix}
                           \widetilde B_{11}&\widetilde B_{12}\\
                           \widetilde B_{21}&\widetilde B_{22} 
                           \end{pmatrix}.
 \]
Here, $\widetilde B_{ij}$ is the Fourier multiplier corresponding to the symbol $b^{ij}$, $1\leq i,j \leq 2,$
 whereby 
 \[ b^{11}_k:=\frac{\widetilde m^{22}_k(\Lambda)}{\widetilde D(k,\Lambda)},\quad b^{12}_k:=-\frac{\widetilde m^{12}_k(\Lambda)}{\widetilde D(k,\Lambda)},\quad 
 b^{21}_k:=-\frac{\widetilde m^{21}_k(\Lambda)}{\widetilde D(k,\Lambda)},\quad b^{22}_k:=\frac{\widetilde m^{11}_k(\Lambda)}{\widetilde D(k,\Lambda)} \qquad\text{for $k\in\mathbb{N}$.}\]
Using now \cite[Theorem 2.1]{JL12}, we see that a Fourier multiplier 
$$\sum_{k\in\mathbb{N}}\alpha_k\cos\left(R_k x\right)\rightarrow \sum_{k\in\mathbb{N}}\lambda_k\alpha_k\cos\left(R_kx\right)$$
belongs to $\mathcal{L}\big(C_{e,per}^{2+\alpha}(\mathbb{R}), C_{e,per}^{3+\alpha}(\mathbb{R})\big)$ if 
\[\text{$\sup_{k\in\mathbb{N}}\vert k \lambda_k\vert<\infty$\qquad\text{and}\qquad $\sup_{k\in\mathbb{N}} k^2\vert \lambda_{k+1}-\lambda_k\vert <\infty$.}\]
Because of this, it is a matter of direct computation to see that $\widetilde B_{ij}\in \mathcal{L}\big(C^{2+\alpha}_{e,per}(\mathbb{R}), C^{3+\alpha}_{e,per}(\mathbb{R})\big) $ for all $1\leq i,j\leq 2.$
Hence, $T$ is an isomorphism, and therefore $\partial_{(f,h)}\Phi(\Lambda,0)$ is a Fredholm operator of index zero.
\end{proof}

Because of Lemma \ref{L:FP} $(i)$ it is clear that we cannot apply the Crandall-Rabinowitz bifurcation theorem  at  $(\Lambda,0)$ with $\Lambda\in\{0,\gamma_2d_2\}.$
As a consequence of this, the laminar flows from which we show that non laminar waves bifurcate 
will not possess stagnation points at the wave surface or on the interface separating the two layers of constant vorticity, cf. \eqref{top}-\eqref{bed}, but only inside the layers. 
This is different than in the case of internal waves propagating between two layers of constant but different density, where in  the presence of capillarity stagnation points may be located also on the internal wave, cf. \cite{Ma12a}.

It is now evident that potential bifurcation values for $\Lambda\not\in\{0,\gamma_2d_2\}$ are to be looked for among the solutions of 
\begin{equation}\label{DR}
D(k,\Lambda)=0
\end{equation}
for some integer $k\geq 1$.
Since in Theorem \ref{CR} the Fr\'echet derivative  $\partial_{(f,h)}\Phi(\Lambda,0)$ needs to be a Fredholm operator of index zero with a one-dimensional kernel,
we need to find  $\Lambda $ such that  \eqref{DR} has exactly one root $1\leq k\in \mathbb{N}.$ 
Plugging the expressions \eqref{a11}-\eqref{a22} in \eqref{DR}, we rediscover the dispersion relation 
\begin{align}\label{DRR}
\Lambda^3 &-\frac{1}{R_k}\Big[\gamma_2\Big(R_k d_2+\frac{\sinh(R_k d_2)\cosh(R_k d_1)}{\cosh(R_k d)}\Big)
+\gamma_1 \frac{\sinh(R_k d_1)\cosh(R_k d_2)}{\cosh(R_k d)}\Big]\Lambda^2 \nonumber \\
&+\tanh(R_k d)\Big[\frac{\gamma_2^{2}d_2-g}{R_k}+ \gamma_2(\gamma_1 -\gamma_2) 
\frac{\sinh(R_k d_1)\sinh(R_k d_2)}{R_k^2\sinh(R_k d)}\Big]\Lambda \nonumber\\
&+g\frac{\tanh(R_k d)}{R_k^2}\Big[(\gamma_1 -\gamma_2)
\frac{\sinh(R_k d_1)\sinh(R_k d_2)}{\sinh(R_k d)}+\gamma_2 d_2R_k\Big]=0,
 \end{align}
found also in  \cite[Equation (5.11)]{CM14} (with $\sigma=0$).
 This relation has been analyzed in the setting of flows without stagnation points in \cite{AC12a} for $\gamma_1\neq 0=\gamma_2$ and in \cite{CoSt11} for $\gamma_1=0\neq\gamma_2$.
Herein, we assume only that $\gamma_1\neq\gamma_2$ and restrict the analysis to the complementary case when stagnation points are included.

In studying the dispersion relation \eqref{DRR} we will make use of the following remark, which allows us to restrict our attention to a few of relevant cases, 
the remaining ones being analogous.
 \begin{rem}\label{R:R}
Note that \eqref{DRR} possesses the following symmetry property: $k$ is a solution of \eqref{DRR} for some $\Lambda\not\in\{0,\gamma_2d_2\}$ and  $(\gamma_1,\gamma_2)\in\mathbb{R}^2$ if and only 
 if $k$ is a solution of \eqref{DRR} for  $-\Lambda\not\in\{0,-\gamma_2d_2\}$  and $(-\gamma_1,-\gamma_2)\in\mathbb{R}^2.$
 Because additionally the inequalities \eqref{top} and \eqref{bed} are invariant under the transformation $(\Lambda,(\gamma_1,\gamma_2))\mapsto (-\Lambda,(-\gamma_1,-\gamma_2)),$
 we are left  only with the two cases:
 \begin{itemize}
  \item[$(i)$]   $\gamma_2>0$  and $\gamma_1\neq \gamma_2;$
  \item[$(ii)$] $\gamma_2= 0$ and $\gamma_1< 0$.
 \end{itemize}
  \end{rem}

\section{Analysis of the dispersion relation: the case $\gamma_2>0$ and $\gamma_1\neq \gamma_2$} \label{Sec:3}
Because the dispersion relation is highly nonlinear in $k$, the study of the roots of \eqref{DRR} when keeping $\Lambda$ fixed seems to be very difficult.
Therefore, we consider the  inverse problem of determining the zeros $\Lambda_1,$ $\Lambda_2,$ $\Lambda_3$ of this cubic equation when keeping $k$ fixed, and then to study the properties of the mappings $k\mapsto \Lambda_i(k)$, $1\leq i\leq 3.$  
We will do this for small wavelength $L,$ because then we can use asymptotic expansions and Cardano's formula in order to determine the roots $\Lambda_i$ of \eqref{DRR}. 
This small wavelength regime corresponds to the setting  $t\to\infty,$
where
\[
t:=R_k=\frac{2\pi k}{L}\in\mathbb{R}.
\]
Plugging $t$ for $R_k$, the equation \eqref{DRR} can be written in the more concise form
\begin{equation}\label{CF}
 \Lambda^3+A(t)\Lambda^2+B(t)\Lambda+C(t)=0.
\end{equation}
We will show in the sequel that equation \eqref{CF} has three real roots when $t$ is sufficiently large.
To this end, we note first that the coefficient functions $A=A(t), $ $B=B(t),$ and $C=C(t)$ and their first derivatives have the following asymptotic expansions for $t\to\infty$:
\begin{equation}\label{A}
\begin{array}{lll}
 &\displaystyle A=-\gamma_2 d_2-\frac{\gamma_1 +\gamma_2}{2t}+o\Big(\frac{1}{t^2}\Big),&\displaystyle A'=\frac{\gamma_1 +\gamma_2}{2}\cdot\frac{1}{t^2}+o\Big(\frac{1}{t^3}\Big),\\[2ex]
 &\displaystyle B=\frac{\gamma_2^2 d_2-g}{t}+\frac{\gamma_2 (\gamma_1 -\gamma_2)}{2t^2}+o\Big(\frac{1}{t^3}\Big),
 &\displaystyle
 B'=-\frac{\gamma_2^2 d_2-g}{t^2}-\frac{\gamma_2 (\gamma_1 -\gamma_2)}{t^3}+ o\Big(\frac{1}{t^4}\Big),\\[2ex]
 &\displaystyle C=\frac{g\gamma_2 d_2}{t}+\frac{g(\gamma_1 -\gamma_2)}{2t^2}+o\Big(\frac{1}{t^3}\Big),&\displaystyle C'=-\frac{g\gamma_2 d_2}{t^2}-\frac{g(\gamma_1 -\gamma_2)}{t^3}+ o\Big(\frac{1}{t^4}\Big).
\end{array}
\end{equation}
Letting $z:=\Lambda +A/3$, we find that $z$ solves the depressed cubic equation 
\begin{equation}\label{redcu}
z^3+pz+q=0,
\end{equation}
with
\begin{align*}
 \frac{p}{3}=\frac{B}{3}-\frac{A^2}{9}=-\frac{(\gamma_2 d_2)^2}{9}+\frac{\gamma_2 d_2 (2\gamma_2 -\gamma_1)-3g}{9t}+\frac{4\gamma_1 \gamma_2 -7\gamma_2^2-\gamma_1^2}{36t^2}+o\Big(\frac{1}{t^3}\Big)
\end{align*}
and 
\begin{align*} 
 \frac{q}{2}=&\frac{A^3}{27}-\frac{AB}{6}+\frac{C}{2}\nonumber\\
 =&-\frac{(\gamma_2 d_2)^3}{27}+\gamma_2 d_2\frac{6g+\gamma_2 d_2(2\gamma_2 -\gamma_1)}{18t}-\Big[\gamma_2 d_2\frac{\gamma_1 ^2-\gamma_1\gamma_2 +\gamma_2^2}{36}+\frac{g(2\gamma_2 -\gamma_1)}{3}\Big]
 \frac{1}{t^2}\nonumber\\
& +\frac{9\gamma_2 (\gamma_1 ^2-\gamma_2 ^2)-(\gamma_1 +\gamma_2)^3}{216}\frac{1}{t^3}+o\Big(\frac{1}{t^4}\Big).
\end{align*}
Observe that the discriminant for \eqref{redcu} is
\[
D:=\Big(\frac{p}{3}\Big)^3+\Big(\frac{q}{2}\Big)^2=-\frac{9g(\gamma_2d_2)^4}{243t}+O\Big(\frac{1}{t^2}\Big)<0 \qquad\text{for $t\to\infty$},
\]
property which implies, cf. \cite{Ti01}, that \eqref{redcu}, and hence also \eqref{CF}, has three real roots.
They are given by the relation $z=r\cos(\beta),$
whereby 
\begin{align}\label{rt} 
r=\sqrt{\frac{-4p}{3}}=\frac{2\gamma_2d_2}{3}-\frac{\gamma_2 d_2(2\gamma_2-\gamma_1)-3g}{3\gamma_2d_2t}+O\Big(\frac{1}{t^2}\Big)
\end{align}
and $\beta$ is one of   the solution of 
\begin{equation*}
 \cos(3\beta)=-\frac{q}{2}\sqrt{-\frac{27}{p^3}}=1-\frac{3^3g}{2(\gamma_2d_2)^2t}+O\Big(\frac{1}{t^2}\Big).
\end{equation*}
Thus, choosing $ \beta:=3^{-1}\arccos\big(-(q/2)\sqrt{-27/p^3}\big)$ we see that $\beta(t)\searrow_{t\to\infty} 0$ and   the roots of \eqref{CF} are
\begin{equation}\label{Lambdas}
 \begin{aligned}
& { \Lambda_1 =r\cos(\beta)-\frac{A}{3}},\\
&\Lambda_2 =r\cos\Big(\beta-\frac{2\pi}{3}\Big){-\frac{A}{3}} =-r\cos\Big(\beta+\frac{\pi}{3}\Big)-\frac{A}{3},\\
&\Lambda_3=r\cos\Big(\beta+\frac{2\pi}{3}\Big){-\frac{A}{3} } =-r\cos\Big(\beta-\frac{\pi}{3}\Big)-\frac{A}{3}. 
 \end{aligned} 
\end{equation}
Together with \eqref{A} and \eqref{rt}, it follows at once that for $t\to\infty$ we have
\begin{equation}\label{limlambda1}
\Lambda_1(t)\to\gamma_2 d_2,\qquad \Lambda_2(t)\to 0,\qquad \Lambda_3(t)\to 0.
\end{equation}
Let us also observe that since $C(t)>0$ for $t\to\infty,$ it must hold that $\Lambda_2(t)\Lambda_3(t)<0$ for $t\to\infty.$
Moreover, it is clear from \eqref{Lambdas} that $\Lambda_2(t)>\Lambda_3(t),$ hence $$\Lambda_2(t)>0>\Lambda_3(t)\qquad\text{for $t\to\infty$.}$$

\subsection{Existence of water flows  bifurcating from $\Lambda_1$}
In order to consider the bifurcation problem for \eqref{NP}, we need to study first the properties of the mapping $[t\mapsto\Lambda_1(t)].$
\begin{lem}\label{Lm}
There is a constant $t_0\geq0$  such that the function $$[[t_0,\infty)\ni t\mapsto\Lambda_1(t)\in(0,\infty)]$$ is strictly monotone.
\end{lem}
\begin{proof}
Let  $\phi(t,\Lambda):=\Lambda^3+A(t)\Lambda^2+B(t)\Lambda+C(t) $ for $\Lambda\in\mathbb{R}$ and $t\geq0.$  
Since for $t\to\infty$ we have 
\begin{equation}\label{fLambda}
 \phi_{\Lambda}(t,\Lambda_1 (t))=(\Lambda_1(t)-\Lambda_2(t))(\Lambda_1(t)-\Lambda_3(t))>0,
\end{equation}
we conclude that $\Lambda_1$ is differentiable with respect to $t.$
On the other hand
$$t^2\Phi_t(t,\Lambda_1(t))=t^2\big(\Lambda_1^2(t) A^{\prime}(t)+\Lambda_1(t) B^{\prime}(t)+C^{\prime}(t)\big) \to_{t\to\infty}\frac{(\gamma_2 d_2)^2(\gamma_1 -\gamma_2)}{2}.$$
Since $\Lambda_1'(t)= -\phi_t(t,\Lambda_1(t))/\phi_{\Lambda}(t,\Lambda_1 (t)),$ we conclude that $\Lambda_1'$ has the same sign as $\gamma_2-\gamma_1.$
The constant $t_0$ is defined as $t_0:=\inf\{t>0\,:\, \text{$|\Lambda_1'|>0$ on $(t,\infty)$}\}.$
\end{proof}

From Lemma \ref{Lm} it follows at once that
\begin{itemize}
 \item if $ \gamma_1<\gamma_2,$ then $\Lambda_1(t)$ satisfies \eqref{top} for $t\geq t_0$;
 \item if $\gamma_1>\gamma_2,$ then $\Lambda_1(t)$ satisfies \eqref{bed} for $t\geq t_0$.
\end{itemize}

We look now for bifurcation solutions when choosing $\Lambda_1$ as the bifurcation point.
Therefore, we choose $t_0>0$ in Lemma \ref{Lm} large enough to guarantee additionally that 
\begin{equation}\label{add}
\begin{aligned}
 &\inf_{[t_0,\infty)}\Lambda_1^2>\sup_{[t_0,\infty)}\big(\Lambda_2^2+\Lambda_3^2\big),\\
 & D(0,\Lambda_1(t))\neq0\qquad\text{for all $t\geq t_0.$}
 \end{aligned}
\end{equation}
Let \begin{equation}\label{L0a}
 L_0:=2\pi/t_0,    
    \end{equation}
fix $L\leq L_0,$ and set $\Lambda_1:=\Lambda_1(2\pi/L).$
Then, since $\phi(2\pi/L, \Lambda_1)=0$, we get
$D(1,\Lambda_1)=0.$
Due to  the choice of $t_0$, the equation $D(\cdot, \Lambda_1)=0$ has no solutions $k\in\mathbb{N}$  other than $k=1.$
Consequently, since $\Lambda_1\not\in\{0,\gamma_2d_2\},$  the derivative 
$\partial_{(f,h)}\Phi_1(\Lambda_1,0)$ is a Fredholm operator with a one-dimensional kernel
\begin{equation}\label{ke1}
 {\rm Ker\,}\partial_{(f,h)}\Phi_1(\Lambda_1,0)={\rm span\, }\big\{ \big(m^{22}_1(\Lambda_1),-m^{21}_1(\Lambda_1)\big)\cos(2\pi x/L)\big\}.
\end{equation}
In order to apply Theorem \ref{CR} to this particular setting, it remains to study whether the transversality condition  is satisfied.
To this end, we obtain the following characterization of ${\rm Im\,} \partial_{(f,h)}\Phi_1(\Lambda_1,0)$.
\begin{lem}\label{imagechar}
Let  $L_0$ be given by \eqref{L0a}, $L\leq L_0$, and set $\Lambda_1:=\Lambda_1(2\pi/L).$
Then, we have
\begin{equation}\label{im}
 {\rm Im\,} \partial_{(f,h)}\Phi(\Lambda_1,0)=\Big\{(\xi,\eta)=\Big( \sum_{k\in\mathbb{N}}\xi_k \cos(R_kx),\sum_{k\in\mathbb{N}}\eta_k \cos(R_kx) \Big)\,:\, \gamma_1=\frac{m^{11}_1(\Lambda_1)}{m^{21}_1(\Lambda_1)}\eta_1\Big\}.
\end{equation}
\end{lem}
\begin{proof}
To prove the claim,  let $(f,h)=\left( \sum_{k\in\mathbb{N}}a_k \cos\left(R_kx\right),\sum_{k\in\mathbb{N}}b_k \cos\left(R_kx\right) \right)$ be such that
$\partial_{(f,h)}\Phi(\Lambda_1,0)(f,h)=(\xi,\eta)$. Then, obviously
\begin{equation*}
\left\{\begin{array}{c}
        m^{11}_1(\Lambda_1)  a_1+m^{12}_1(\Lambda_1) b_1=\gamma_1,\\[1ex]
        m^{21}_1(\Lambda_1) a_1+m^{22}_1(\Lambda_1) b_1=\eta_1.
        \end{array}\right. 
\end{equation*}
Because $D(1,\Lambda_1)=0,$ we find from \eqref{a11}-\eqref{a22} that $$\frac{m^{11}_1(\Lambda_1)}{m^{21}_1(\Lambda_1)}=\frac{m^{12}_1(\Lambda_1)}{m^{22}_1(\Lambda_1)}=:\mu\neq0.$$
Hence, $(\xi,\eta)$ is an element of the set defined by the right-hand side of \eqref{im}. 
Because the latter set is a closed  subspace of $\big(C^{2+\alpha}_{e,per}(\mathbb{R})\big)^2$ of codimension one that contains ${\rm Im\,} \partial_{(f,h)}\Phi(\Lambda_1,0),$ the conclusion follows from Lemma \ref{L:FP}.
\end{proof}
We are now at the point of showing the transversality condition $(c)$ from Theorem \ref{CR}.
\begin{lem}\label{L:trans}
We have that
$$\partial_{\Lambda  (f,h)}\Phi(\Lambda_1,0) \big[ \big(m^{22}_1(\Lambda_1),-m^{21}_1(\Lambda_1)\big)\cos(2\pi x/L)  \big]\notin{\rm Im\,} \partial_{(f,h)}\Phi(\Lambda_1,0).$$
\end{lem}
\begin{proof}
Since for $a,b\in\mathbb{R},$ it holds that
\begin{align*}
&\partial_{\Lambda (f,h)}\Phi(\Lambda_1,0) \big[ (a,b)\cos(2\pi x/L)  \big]\\[1ex]
&=\big(a m^{11}_{1,\Lambda}(\Lambda_1)+bm^{12}_{1,\Lambda}(\Lambda_1), a m^{21}_{1,\Lambda}(\Lambda_1)+bm^{22}_{1,\Lambda}(\Lambda_1)\big)\cos(2\pi x/L), 
 \end{align*}
 we are left to show that 
 \begin{align*}
 m^{22}_1(\Lambda_1) m^{11}_{1,\Lambda}(\Lambda_1)-m^{21}_1(\Lambda_1)m^{12}_{1,\Lambda}(\Lambda_1)\neq
  \frac{m^{11}_1(\Lambda_1)}{ m^{21}_1(\Lambda_1)} \big(m^{22}_1(\Lambda_1) m^{21}_{1,\Lambda}(\Lambda_1)-m^{21}_1(\Lambda_1)m^{22}_{1,\Lambda}(\Lambda_1)\big),
 \end{align*}
or equivalently that
 \begin{align*}
m^{22}_1(\Lambda_1) m^{11}_{1,\Lambda}(\Lambda_1)-m^{21}_1(\Lambda_1)m^{12}_{1,\Lambda}(\Lambda_1) \neq m^{12}_1(\Lambda_1)  m^{21}_{1,\Lambda}(\Lambda_1)-m^{11}_1(\Lambda_1)m^{22}_{1,\Lambda}(\Lambda_1).
 \end{align*}
 Hence, we need to show that $ D_\Lambda(1,\Lambda_1)\neq0.$
 Recalling the definition of the mapping $\phi$ from the proof of Lemma \ref{Lm}, we have that $D(1,\Lambda)=\phi(2\pi/L,\Lambda),$
 and therefore
 \[
  D_\Lambda(1,\Lambda_1)=\Phi_\Lambda (2\pi/L,\Lambda_1(2\pi/L))>0,
 \]
 which is the desired property.
\end{proof}

\begin{thm}[Bifurcation from $\Lambda_1$]\label{MT1}
Let $\gamma_2>0,$  $\gamma_1\neq\gamma_2$ and let $\alpha\in(0,1)$ be given.
Furthermore, let $L_0$ be the constant defined by  \eqref{L0a}  and $L\leq L_0.$ 
Then, there exists a real-analytic curve $(\Lambda,(f,h)):(-\varepsilon,\varepsilon)\to (0,\infty)\times \mathcal{O}$ consisting only of solutions of the problem \eqref{NP}.
This curve contains exactly one trivial solution of \eqref{NP}, and for $s\to0$ we have that
\[\Lambda(s)=\Lambda_1+O(s),\qquad (f,h)(s)=s  \big(m^{22}_1(\Lambda_1),-m^{21}_1(\Lambda_1)\big)\cos(2\pi x/L)+O(s^2),\]
whereby $\Lambda_1:=\Lambda_1(2\pi/L).$
The flow determined by $(\Lambda(s),(f,h)(s)), s\in(-\varepsilon,\varepsilon),$ contains a critical layer consisting  of closed streamlines   very close to the internal wave 
\begin{itemize}
 \item[$(i)$] in the layer adjacent to the wave surface  if $\gamma_1<\gamma_2,$ or
 \item[$(ii)$] in the bottom layer if $\gamma_1>\gamma_2.$ 
\end{itemize}
Moreover, the amplitude of the internal wave    is much larger than that of  the surface wave, cf. Figure \ref{Fig1}.
\end{thm}
\begin{proof}
 It remains only to show that the amplitude of the internal wave is much larger than that of the surface wave.
To this  end, we note that due to $D(1,\Lambda_1)=0,$ we have 
\begin{equation*}
 -\frac{m^{22}_1(\Lambda_1)}{m^{21}_1(\Lambda_1)} = -\frac{m^{12}_1(\Lambda_1)}{m^{11}_1(\Lambda_1)}\to_{L\to 0}{\rm sign\,}(\gamma_1-\gamma_2)\infty,
\end{equation*}
since
\begin{align*}
 \lim_{L\to 0} \frac{m^{12}_1(\Lambda_1)}{m^{11}_1(\Lambda_1)} 
 =&\lim_{t\to\infty}\frac{g+\gamma_2\Lambda_1(t)-\Lambda^2_1(t) \displaystyle\frac{t}{\tanh\left(td_2\right)}}{- \Lambda_1(t)(\gamma_2 d_2-\Lambda_1(t))\displaystyle \frac{t}{\sinh\left(td_2\right)}}={\rm sign\,}(\gamma_2-\gamma_1)\infty.
\end{align*}
\end{proof}

\begin{figure}[t2]
\includegraphics[width=2.2in, angle=270]{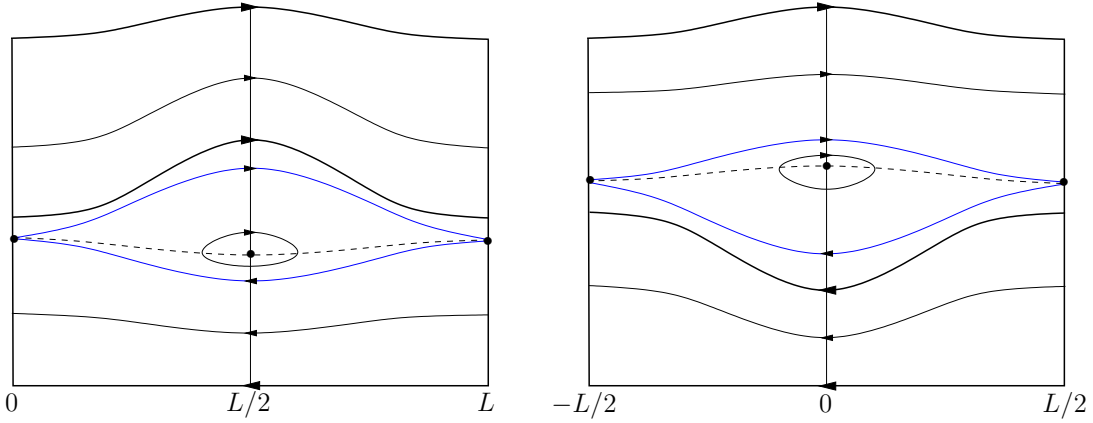}
\caption[Stream pattern]{\small This figure illustrates the streamlines in the moving frame for the solutions found in Theorem \ref{MT1} for $\gamma_1>\gamma_2>0$ (left) and  $\gamma_1<\gamma_2, \gamma_2>0$  (right), cf. Lemmas \ref{fig1_left}-\ref{fig1_right}.
The thick  streamlines represent the wave surface, the internal wave, and the flat bad, respectively.
  The blue streamlines are separatrices which bound the critical layer and the dashed line consists of points where the $y-$derivative of the stream function vanishes.
This line contains in both cases exactly three stagnation points: two located at $x=0$ and $x=L$, and a third one inside the critical layer at $x=L/2.$}\label{Fig1}
\end{figure}

\subsection{Existence of water flows  bifurcating from $\Lambda_2$}
For $t\to\infty$ we have that
\begin{itemize}
 \item $\Lambda_2(t)$ satisfies \eqref{top};
 \item  if $\gamma_1d_1+\gamma_2 d_2\leq 0,$ then $\Lambda_2(t)$ satisfies also \eqref{bed}.
\end{itemize}
Letting $\phi=\phi(t,\Lambda)$ be the function defined in the proof of Lemma \ref{Lm}, we note that for large $t$ we have
\begin{equation*} 
 \phi_{\Lambda}(t,\Lambda_2 (t))=(\Lambda_2(t)-\Lambda_1(t))(\Lambda_2(t)-\Lambda_3(t))<0.
\end{equation*}
Hence, $\Lambda_2$ is differentiable with respect to $t.$
Moreover, it follows from \eqref{A} that
$$t^2\phi_t(t,\Lambda_2(t))\to_{t\to\infty}-g\gamma_2d_2,$$
and therefore $\Lambda_2'(t)= -\phi_t(t,\Lambda_2(t))/\phi_{\Lambda}(t,\Lambda_2 (t))<0$ when $t$ is large.
Defining $$t_0:=\inf\{t>0\,:\, \text{$\Lambda_2'<0$ on $(t,\infty)$}\},$$
we see that $[[t_0,\infty)\ni t\mapsto\Lambda_2(t)\in(0,\infty)]$ is  decreasing.
In view of this property, we can choose $t_0>0$   large enough to guarantee  that 
\begin{equation}\label{add1}
\begin{aligned}
 &\sup_{[t_0,\infty)}\Lambda_2<\inf_{[t_0,\infty)}\Lambda_1,\\
 &D(0,\Lambda_2(t))\neq0\qquad\text{for all $t\geq t_0.$}
 \end{aligned}
\end{equation}
Then, we set   
\begin{equation}\label{L0a1}
 L_0:=2\pi/t_0,    
    \end{equation}
we fix $L\leq L_0,$ and define $\Lambda_2:=\Lambda_2(2\pi/L).$
Recalling that $\phi(2\pi/L, \Lambda_2)=0$, we obtain that
$D(1,\Lambda_2)=0.$
In fact,  the equation $D(\cdot, \Lambda_2)=0$ has  $k=1$ as the only integer solution, cf. \eqref{add1}.
Because  of $\Lambda_2\in(0,\gamma_2d_2)$,
$\partial_{(f,h)}\Phi_1(\Lambda_2,0)$ is a Fredholm operator with a one-dimensional kernel
\begin{equation*}
 {\rm Ker\,}\partial_{(f,h)}\Phi_1(\Lambda_2,0)={\rm span\, }\big\{ \big(m^{22}_1(\Lambda_2),-m^{21}_1(\Lambda_2)\big)\cos(2\pi x/L)\big\}.
\end{equation*}
Using the same arguments as in the proof of Lemma \ref{imagechar}, we see that 
\begin{equation*} 
 {\rm Im\,} \partial_{(f,h)}\Phi(\Lambda_2,0)=\Big\{(\xi,\eta)=\Big( \sum_{k\in\mathbb{N}}\xi_k \cos(R_kx),\sum_{k\in\mathbb{N}}\eta_k \cos(R_kx) \Big)\,:\, \gamma_1=\frac{m^{11}_1(\Lambda_2)}{m^{21}_1(\Lambda_2)}\eta_1\Big\}.
\end{equation*}
Moreover, the transversality condition   
$$\partial_{\Lambda  (f,h)}\Phi(\Lambda_2,0) \big[ \big(m^{22}_1(\Lambda_2),-m^{21}_1(\Lambda_2)\big)\cos(2\pi x/L)  \big]\notin{\rm Im\,} \partial_{(f,h)}\Phi(\Lambda_2,0)$$
reduces to showing that
$ D_\Lambda(1,\Lambda_2)=\phi_\Lambda (2\pi/L,\Lambda_2(2\pi/L))\neq0,$
relation which holds true.
 We conclude with the following result.
 
 \begin{figure}[ht]
\includegraphics[width=2.4in, angle=270]{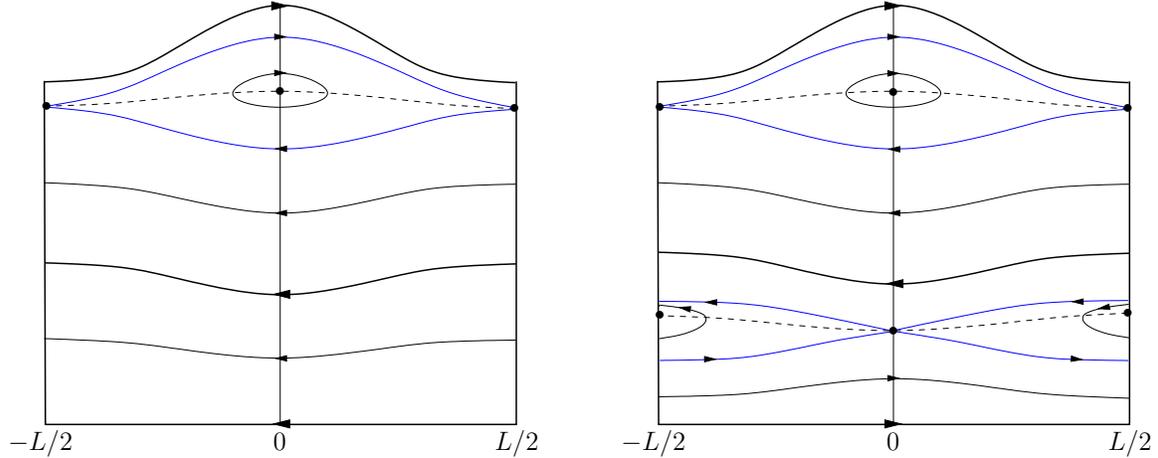}
\caption[Stream pattern]{\small This figure illustrates the streamlines in the moving frame for the solutions found in Theorem \ref{MT2} for $\gamma_1d_1+\gamma_2d_2>0$ (left)
and  $\gamma_1d_1+\gamma_2d_2\leq0$  (right), cf. Lemmas \ref{fig2_left}-\ref{fig2_right}.
}\label{Fig2}
\end{figure}

\begin{thm}[Bifurcation from $\Lambda_2$]\label{MT2}
Let $\gamma_2>0,$  $\gamma_1\neq\gamma_2$ and let $\alpha\in(0,1)$ be given.
Furthermore, let $L_0$ be the constant defined by  \eqref{L0a1}  and $L\leq L_0.$ 
Then, there exists a real-analytic curve $(\Lambda,(f,h)):(-\varepsilon,\varepsilon)\to (0,\infty)\times \mathcal{O}$ consisting only of solutions of the problem \eqref{NP}.
This curve contains exactly one trivial solution of \eqref{NP}, and for $s\to0$ we have that
\[\Lambda(s)=\Lambda_2+O(s),\qquad (f,h)(s)=s  \big(m^{22}_1(\Lambda_2),-m^{21}_1(\Lambda_2)\big)\cos(2\pi x/L)+O(s^2),\]
whereby $\Lambda_2:=\Lambda_2(2\pi/L).$
The flow determined by $(\Lambda(s),(f,h)(s)), s\in(-\varepsilon,\varepsilon),$ contains a critical layer consisting  of closed streamlines 
\begin{itemize}
 \item[$(i)$] in the layer adjacent to the wave surface if $\gamma_1d_1+\gamma_2d_2>0$;
 \item[$(ii)$] in each of the  layers if $\gamma_1d_1+\gamma_2d_2\leq0.$ 
\end{itemize}
The vortex in the top layer is located right beneath the wave surface. 
Moreover, the amplitude of the internal wave  between the two layers is much smaller than that of  the surface wave, cf. Figure \ref{Fig2}.
\end{thm}
\begin{proof}
 It remains only to show that the amplitude of the surface wave is much larger than that of the internal wave.
This follows from \eqref{a21}-\eqref{a22}, as we have
\begin{equation*}
 -\frac{m^{21}_1(\Lambda_2)}{m^{22}_1(\Lambda_2)}\to_{L\to 0}-\infty.
\end{equation*}
\end{proof}

\subsection{Existence of water flows  bifurcating from $\Lambda_3$}
Since $0>\Lambda_3(t)\to_{t\to\infty}0$ we see that
 $\Lambda_3(t)$ satisfies \eqref{bed} provided that  $\gamma_1d_1+\gamma_2 d_2< 0.$ 
Because for large $t$ 
\begin{equation*} 
 \phi_{\Lambda}(t,\Lambda_3(t))=(\Lambda_3(t)-\Lambda_1(t))(\Lambda_3(t)-\Lambda_2(t))>0,
\end{equation*}
the function $\Lambda_3$ is differentiable with respect to $t.$
Since
$t^2\phi_t(t,\Lambda_3(t))\to_{t\to\infty}-g\gamma_2d_2,$
we conclude that  $\Lambda_3'(t)>0$ when $t$ is large.
Defining $t_0:=\inf\{t>0\,:\, \text{$\Lambda_3'>0$ on $(t,\infty)$}\},$
we see that $[[t_0,\infty)\ni t\mapsto\Lambda_3(t)\in(-\infty,0)]$ is increasing.
In view of this property, we can choose $t_0>0$   large enough to guarantee  that 
\begin{equation}\label{add2}
D(0,\Lambda_3(t))\neq0\qquad\text{for all $t\geq t_0.$}
\end{equation}
Let 
\begin{equation}\label{L0a2}
 L_0:=2\pi/t_0,    
    \end{equation}
choose $L\leq L_0,$ and define $\Lambda_3:=\Lambda_3(2\pi/L).$
Since $\phi(2\pi/L, \Lambda_3)=0$, we get that 
$k\in \mathbb{N}$ solves $D(k,\Lambda_3)=0$ if and only if $k=1.$
Moreover, since $\Lambda_3<0, $
$\partial_{(f,h)}\Phi_1(\Lambda_3,0)$ is a Fredholm operator with a one-dimensional kernel
\begin{equation*}
 {\rm Ker\,}\partial_{(f,h)}\Phi_1(\Lambda_3,0)={\rm span\, }\big\{ \big(m^{22}_1(\Lambda_3),-m^{21}_1(\Lambda_3)\big)\cos(2\pi x/L)\big\}.
\end{equation*}
As in Lemma \ref{imagechar}, we find that 
\begin{equation*} 
 {\rm Im\,} \partial_{(f,h)}\Phi(\Lambda_3,0)=\Big\{(\xi,\eta)=\Big( \sum_{k\in\mathbb{N}}\xi_k \cos(R_kx),\sum_{k\in\mathbb{N}}\eta_k \cos(R_kx) \Big)\,:\, \gamma_1=\frac{m^{11}_1(\Lambda_3)}{m^{21}_1(\Lambda_3)}\eta_1\Big\},
\end{equation*}
 the transversality condition   
$$\partial_{\Lambda  (f,h)}\Phi(\Lambda_3,0) \big[ \big(m^{22}_1(\Lambda_3),-m^{21}_1(\Lambda_3)\big)\cos(2\pi x/L)  \big]\notin{\rm Im\,} \partial_{(f,h)}\Phi(\Lambda_3,0)$$
being equivalent to
$  D_\Lambda(1,\Lambda_3)=\phi_\Lambda (2\pi/L,\Lambda_3(2\pi/L))\neq0.$
This shows that all the assumptions of Theorem \ref{CR} are satisfied.
Consequently, we have the following result.
 
\begin{thm}[Bifurcation from $\Lambda_3$]\label{MT3}
Let $\gamma_2>0,$    $\alpha\in(0,1)$, and assume  $\gamma_1d_1+\gamma_2d_2<0$.
Furthermore, let $L_0$ be the constant defined by  \eqref{L0a2}  and $L\leq L_0.$ 
Then, there exists a real-analytic curve $(\Lambda,(f,h)):(-\varepsilon,\varepsilon)\to (0,\infty)\times \mathcal{O}$ consisting only of solutions of the problem \eqref{NP}.
This curve contains exactly one trivial solution of \eqref{NP}, and for $s\to0$ we have that
\[\Lambda(s)=\Lambda_3+O(s),\qquad (f,h)(s)=s  \big(m^{22}_1(\Lambda_3),-m^{21}_1(\Lambda_3)\big)\cos(2\pi x/L)+O(s^2),\]
whereby $\Lambda_3:=\Lambda_3(2\pi/L).$
The flow determined by $(\Lambda(s),(f,h)(s)), s\in(-\varepsilon,\varepsilon),$ contains a critical layer consisting  of closed streamlines 
in the layer adjacent to the bed.
Moreover, the amplitude of the internal wave  between the two layers is much smaller than that of  the surface wave, cf. Figure \ref{Fig3}.
\end{thm}
\begin{proof}
The claim concerning the amplitude of the surface and internal waves follows from \eqref{a21}-\eqref{a22}, as we have
\begin{equation*}
 -\frac{m^{21}_1(\Lambda_2)}{m^{22}_1(\Lambda_2)}\to_{L\to 0}\infty.
\end{equation*}
\end{proof}

 \begin{figure}[ht]
\includegraphics[width=2.4in, angle=270]{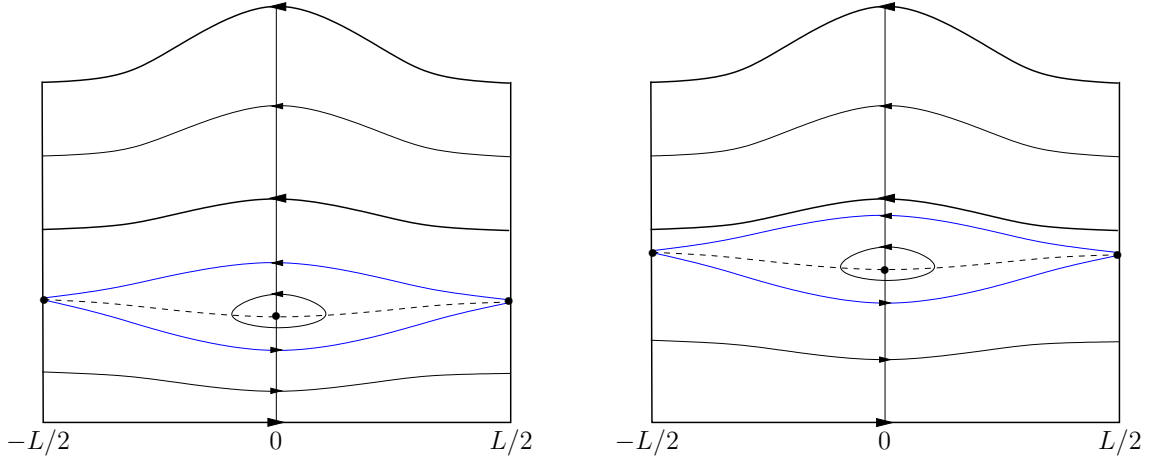}
\caption[Stream pattern]{\small This figure illustrates the streamlines in the moving frame for the solutions found in Theorem \ref{MT3} (left) and Theorems \ref{MT4} and \ref{MT5} (right), cf. Lemmas \ref{fig3_left}-\ref{fig3_right}.}\label{Fig3}
\end{figure}

\section{Analysis of the dispersion relation: the case $\gamma_2=0$ and $\gamma_1<0$} \label{Sec:4}
Because of $\gamma_2=0,$ the inequality  \eqref{top} reduces to $\Lambda=0,$ situation when $\partial_{(f,h)}\Phi(\Lambda,0)$ is not
even a Fredholm operator, cf. Lemma \ref{L:FP}.
For this reason we consider the bifurcation problem for \eqref{NP} just for values of $\Lambda$ which satisfy \eqref{bed}.
Hence, the flows that we construct will have stagnation points only in the bottom layer.

With the notation from Section \ref{Sec:3}, we determine for the depressed cubic equation \eqref{redcu} that
\[
D=\Big(\frac{p}{3}\Big)^3+\Big(\frac{q}{2}\Big)^2=-\frac{g^3}{27}t^{-3}+O\big(t^{-4}\big)<0 \qquad\text{for $t\to\infty$},
\]
hence \eqref{CF} has again three positive roots.
They are given by the relation $z=r\cos(\beta),$
whereby 
\begin{align}\label{rt2} 
r=\sqrt{\frac{-4p}{3}}=2 \sqrt{\frac{g}{3}}t^{-1/2}+\frac{\gamma_1^2 }{36}\sqrt{\frac{3}{g}}t^{-3/2}+O\big(t^{-5/2}\big)
\end{align}
and $\beta$ is one of   the solutions of 
\begin{equation*}
 \cos(3\beta)=-\frac{q}{2}\sqrt{-\frac{27}{p^3}}=-\gamma_1 \sqrt{\frac{3}{g}}t^{-1/2}+\frac{10\gamma_1^3 }{72g}\sqrt{\frac{3}{g}}t^{-3/2}+O\big(t^{-5/2}\big).
\end{equation*}
Setting $\beta:=3^{-1}\arccos\big(-(q/2)\sqrt{-27/p^3}\big)$, we see that $\beta(t) \to_{t\to\infty} \pi/6$   and the roots of \eqref{CF} are
\begin{equation}\label{Lambdas1}
 \begin{aligned}
& {\Lambda_1 =r\cos(\beta)-\frac{A}{3}},\\
&\Lambda_2 =r\cos\Big(\beta-\frac{2\pi}{3}\Big){-\frac{A}{3}} =-r\cos\Big(\beta+\frac{\pi}{3}\Big)-\frac{A}{3},\\
&\Lambda_3=r\cos\Big(\beta+\frac{2\pi}{3}\Big){-\frac{A}{3}}  =-r\cos\Big(\beta-\frac{\pi}{3}\Big)-\frac{A}{3}. 
 \end{aligned}
\end{equation}
It follows now easily from \eqref{A}, \eqref{rt2}, and \eqref{Lambdas1} that  $\Lambda_i\to_{t\to\infty}0$ for $i\in\{1,2,3\}$ and that $\Lambda_3<\Lambda_2<0<\Lambda_1$ for $t\to\infty.$
Thus, we can find $t_0>0$ such that 
\begin{equation}\label{Eq1}
\begin{aligned}
 &\gamma_1d_1<\Lambda_3<\Lambda_2<0<\Lambda_1\qquad\text{on $[t_0,\infty)$},\\
&D(0,\Lambda_i(t))\neq0 \qquad\text{for all $t\geq t_0$, $i=2,3$}.
 \end{aligned}
\end{equation}
In view of $\gamma_1<0,$ the relation \eqref{bed} is equivalent to $\Lambda\in(\gamma_1d_1,0)$ and therefore just flows bifurcating from negative $\Lambda$ may contain stagnation points.
For this reason, we only investigate  in the following the functions $\Lambda_2 $ and $\Lambda_3$.
\begin{lem}\label{L:a}
 There exists $t_0>0$   such that \eqref{Eq1} holds and $\Lambda_i:[t_0,\infty)\to(-\infty,0), i=2,3,$ are both increasing functions.
\end{lem}
\begin{proof}
 Note first that $\phi_\Lambda(t,\Lambda_2(t))<0$ and $\phi_\Lambda(t,\Lambda_3(t))>0$ for $t\geq t_0.$
 Hence, $\Lambda_i, i=2,3,$ are differentiable on $[t_0,\infty).$
 Moreover, it is easy to see from \eqref{rt2} and \eqref{A} that
 \[\lim_{t\to\infty}t^{5/2}\phi_t(t,\Lambda_3(t))=\lim_{t\to\infty}t^{5/2}\Lambda_3(t)B'(t)=-g^{3/2}.\]
Hence, we may chose $t_0$ large to ensure  the assertion for the mapping  $\Lambda_3.$
 
 This argument does no longer work for  $\Lambda_2$ as $\cos(\beta+\pi/3)\to_{t\to\infty}0.$
 Hence, we have to determine an expansion for $\cos(\beta+\pi/3).$
 Let 
\[z_0:=\frac{\sqrt{3}}{2}-\frac{\gamma_1}{6}\sqrt{\frac{3}{g}}t^{-1/2},\]
and observe that
\begin{align*}
|\cos(\beta)-z_0||\cos^2(\beta)+z_0\cos(\beta)+z_0^2|=&
 |4\cos^3(\beta)-3\cos(\beta)-4z_0^3+3z_0|\\
 &=\Big|\cos(3\beta)+\gamma_1 \sqrt{\frac{3}{g}}t^{-1/2}\Big|+O(t^{-1})=O(t^{-1}).
\end{align*}
Therewith ${\cos}(\beta)=z_0+O(t^{-1})$ for $t\to\infty.$
It is now easy to see that for $t\to\infty $ we have
\begin{align*}
 \sin(\beta)=&\frac{1}{2}+\frac{\gamma_1}{2\sqrt{g}}t^{-1/2}+O(t^{-1}),
 \end{align*}
 from which it follows  easily that
 \begin{align}
  \Lambda_2(t):= \frac{5\gamma_1}{6}t^{-1}+O(t^{-3/2})\qquad\text{and}\qquad\Lambda_3(t):=-\sqrt{g}t^{-1/2}-\frac{\gamma_1}{6}t^{-1}+O(t^{-3/2}).\label{LL3} 
\end{align}
 The expansion \eqref{LL3} combined with \eqref{A} shows  that
 \[\lim_{t\to\infty}t^3 \phi_t(t,\Lambda_2(t))=\lim_{t\to\infty} t^3\big(\Lambda_2(t)B'(t)+C'(t)\big)=-g\gamma_1/6>0,\]
 relation which proves the claim.
\end{proof}

\begin{thm}[Bifurcation from $\Lambda_3$]\label{MT4}
Let $\gamma_2=0,$ $\gamma_1<0,$ and $\alpha\in(0,1)$.
Furthermore, let $L_0:=2\pi/t_0$  and $L\leq L_0.$ 
Then, there exists a real-analytic curve $(\Lambda,(f,h)):(-\varepsilon,\varepsilon)\to (0,\infty)\times \mathcal{O}$ consisting only of solutions of the problem \eqref{NP}.
This curve contains exactly one trivial solution of \eqref{NP}, and for $s\to0$ we have that
\[\Lambda(s)=\Lambda_3+O(s),\qquad (f,h)(s)=s  \big(m^{22}_1(\Lambda_3),-m^{21}_1(\Lambda_3)\big)\cos(2\pi x/L)+O(s^2),\]
whereby $\Lambda_3:=\Lambda_3(2\pi/L).$
The flow determined by $(\Lambda(s),(f,h)(s)), s\in(-\varepsilon,\varepsilon),$ contains a critical layer consisting  of closed streamlines 
in the layer adjacent to the bed just below the internal wave.
Moreover, the amplitude of the internal wave  between the two layers is much smaller than that of  the surface wave, cf. Figure \ref{Fig3}.
\end{thm}
\begin{proof}
Because $t:=2\pi/L\geq t_0,$ we know from \eqref{Eq1} and Lemma \ref{L:FP} that $\partial_{(f,h)}\Phi_1(\Lambda_3,0)$ is a Fredholm operator. 
To determine its kernel we need to solve $D(k,\Lambda_3)=\phi(kt,\Lambda_3)=0.$
As $\Lambda_3=\Lambda_3(t),$ we see that $D(1,\Lambda_3)=0$, while \eqref{Eq1} ensures that $D(0,\Lambda_3)=0.$ 
Recalling Lemma \ref{L:a}, we see that $D(k,\Lambda_3)\neq 0$ for all $k\geq 2.$
Indeed, for $k\geq 2,$ $\Lambda_3(kt)>\Lambda_3$, and if $\Lambda_3=\Lambda_2(kt)$, then  $\Lambda_3=\Lambda_2(kt)>\Lambda_2(t)>\Lambda_3,$
a contradiction.
Hence,
$\partial_{(f,h)}\Phi_1(\Lambda_3,0)$ is a Fredholm operator with a one-dimensional kernel
\begin{equation*}
 {\rm Ker\,}\partial_{(f,h)}\Phi_1(\Lambda_3,0)={\rm span\, }\big\{ \big(m^{22}_1(\Lambda_3),-m^{21}_1(\Lambda_3)\big)\cos(2\pi x/L)\big\}.
\end{equation*}
Similarly as before we have
\begin{equation*} 
 {\rm Im\,} \partial_{(f,h)}\Phi(\Lambda_3,0)=\Big\{(\xi,\eta)=\Big( \sum_{k\in\mathbb{N}}\xi_k \cos(R_kx),\sum_{k\in\mathbb{N}}\eta_k \cos(R_kx) \Big)\,:\, \gamma_1=\frac{m^{11}_1(\Lambda_3)}{m^{21}_1(\Lambda_3)}\eta_1\Big\},
\end{equation*}
 and one can verify that the transversality condition   
$$\partial_{\Lambda  (f,h)}\Phi(\Lambda_3,0) \big[ \big(m^{22}_1(\Lambda_3),-m^{21}_1(\Lambda_3)\big)\cos(2\pi x/L)  \big]\notin{\rm Im\,} \partial_{(f,h)}\Phi(\Lambda_3,0)$$
is also satisfied. 
We are thus in the position of applying Theorem \ref{CR}.
To finish the proof, we infer from \eqref{a21}-\eqref{a22} and \eqref{LL3}, that
\begin{equation*}
 -\frac{m^{21}_1(\Lambda_3)}{m^{22}_1(\Lambda_3)}\to_{L\to 0}\infty.
\end{equation*}
\end{proof}

When considering bifurcation from $\Lambda_2$ the situation is more complicated because the derivative $\partial_{(f,h)}\Phi_1(\Lambda_2,0)$ {may}
possess a two-dimensional kernel if $\Lambda_3(2\pi k/L)= \Lambda_2(2\pi/L)$ for some $L\geq L_0$ and some integer $k\geq2.$
 When this happens, the integer $k$ is unique, cf. Lemma \ref{L:a}, so that we can conclude the existence of a curve of bifurcating  solutions from Theorem \ref{MT4}.
 When $\Lambda_3(2\pi k/L)\neq\Lambda_2(2\pi/L)$ for  all $k\geq2,$ we can apply again Theorem \ref{CR}.
\begin{thm}[Bifurcation from $\Lambda_2$]\label{MT5}
Let $\gamma_2=0,$ $\gamma_1<0,$ and let $\alpha\in(0,1)$.
Furthermore, let $L_0:=2\pi/t_0$  and $L\leq L_0$.
\begin{itemize}
 \item [$(i)$] Assume that $\Lambda_3(2\pi k/L)\neq \Lambda_2(2\pi/L)$ for all integers $k\geq 2.$ 
Then, there exists a real-analytic curve $(\Lambda,(f,h)):(-\varepsilon,\varepsilon)\to (0,\infty)\times \mathcal{O}$ consisting only of solutions of the problem \eqref{NP}.
This curve contains exactly one trivial solution of \eqref{NP}, and for $s\to0$ we have that
\[\Lambda(s)=\Lambda_2+O(s),\qquad (f,h)(s)=s  \big(m^{22}_1(\Lambda_2),-m^{21}_1(\Lambda_2)\big)\cos(2\pi x/L)+O(s^2),\]
whereby $\Lambda_2:=\Lambda_2(2\pi/L).$
\item [$(ii)$] Assume that $\Lambda_3(2\pi k/L)= \Lambda_2$ for some integer $k\geq 2.$
Then the assertion of Theorem \ref{MT4} holds true, but with $L$ replaced by $L/k.$
\end{itemize}
The flow determined by $(\Lambda(s),(f,h)(s)), s\in(-\varepsilon,\varepsilon),$ contains a critical layer consisting  of closed streamlines 
in the layer adjacent to the bed just beneath the internal wave.
Moreover, the amplitude of the internal wave  between the two layers is much smaller than that of  the surface wave, cf. Figure \ref{Fig3}.
\end{thm}
\begin{proof}
Setting $t:=2\pi/L\geq t_0,$ we know from \eqref{Eq1} and Lemma \ref{L:FP} that $\partial_{(f,h)}\Phi_1(\Lambda_2,0)$ is a Fredholm operator. 
To determine its kernel we need to solve $D(k,\Lambda_2)=\Phi(kt,\Lambda_2)=0.$
A solution of this equation is $k=1$ as $\Lambda_2=\Lambda_2(t).$
Equation \eqref{Eq1} ensures  additionally that $D(0,\Lambda_2)=0.$ 
Because $\Lambda_3$ is increasing to zero, there may exist  a (unique) integer $k\geq 2$ such that $\Lambda_3(2\pi k/L)= \Lambda_2(2\pi/L),$
hence $D(k, \Lambda_2)=0.$ In this case we are in the situation $(ii) $ and the proof is obvious.
If $\Lambda_3(2\pi k/L)\neq \Lambda_2(2\pi/L)$ for all integers $k\geq 2,$ then we are in the case $(i)$ and the proof  is similar to that of Theorem \ref{MT4}.
\end{proof}

 \begin{rem}
Since the properties of the functions $\Lambda_i(t)$, for $i=1,2,3$ were essential in finding the branches of solutions to the water wave problem,
 we summarize them in the Table \ref{T:T1} below.
 \end{rem}
 
 \begin{table}[h]
 \begin{tabular}{| c | c | c | c| }
\hline 
 $-$ & $\begin{array}{lll}\phantom{a}\\[-2ex] \gamma_1>0\\[-2ex]\phantom{a} \end{array}$   &$\gamma_1=0$ & $\gamma_1<0$ \\
\hline
 $\gamma_2>0$ & \multicolumn{3}{|c|}{ $\begin{array}{lll}\phantom{a}\\[-2ex] \Lambda_3<0<\Lambda_2<\Lambda_1\\ \Lambda_1\to \gamma_2d_2,\quad \Lambda_i\to 0,\,i\in\{2,3\}\\[-2ex]\phantom{a} \end{array}$}     \\
 \hline
 $\gamma_2=0$ & $\begin{array}{lll}\phantom{a}\\[-2ex] \Lambda_3<0<\Lambda_2<\Lambda_1<\gamma_1d_1\\ \Lambda_i\to 0,\quad i\in\{1,2,3\}\\[-2ex]\phantom{a} \end{array}$ &  $-$  & 
 $\begin{array}{lll}\phantom{a}\\[-2ex]\gamma_1d_1< \Lambda_3<\Lambda_2<0<\Lambda_1\\ \Lambda_i\to 0,\quad i\in\{1,2,3\} \\[-2ex]\phantom{a}\end{array}$    \\[1ex]
 \hline
 $\gamma_2<0$ &  \multicolumn{3}{|c|}{ $\begin{array}{lll}\phantom{a}\\[-2ex] \Lambda_3< \Lambda_2<0<\Lambda_1\\ \Lambda_3\to \gamma_2d_2,\quad \Lambda_i\to 0,\,i\in\{1,2\}\\[-2ex]\phantom{a} \end{array}$} \\   
\hline
\end{tabular}\\[2ex]
 \caption{Properties of the roots $\Lambda_i, i\in\{1,2,3\},$  of the dispersion relation \eqref{DRR} in dependence of the vorticity constants $\gamma_i, i\in\{1,2\}$ for $\gamma_1\neq\gamma_2 $ and for large $R_k=2\pi k/L.$
 Our analysis is dedicated to the cases:  $(i)$\,\,    $\gamma_2>0$  and $\gamma_1\neq \gamma_2;$ and $(ii)$\,\, $\gamma_2= 0$ and $\gamma_1< 0$.
 The analysis in the other two cases: $(iii)$\,\,    $\gamma_2<0$  and $\gamma_1\neq \gamma_2;$ and $(iv)$\,\, $\gamma_2= 0$ and $\gamma_1> 0$ is similar to that for $(i)$ and $(ii)$, respectively (see Remark \ref{R:R}).}
    \label{T:T1}   
\end{table}


\appendix
\section{}\label{App}
We present herein  the proof of Lemma \ref{L:FD} and additionally we rigorously prove that the streamline pattern for the solutions that we found is as shown in Figures \ref{Fig1}-\ref{Fig3}, respectively. 
To this end, we determine first explicit expressions for the elliptic and  boundary operators introduced right before Observation \ref{Obs:1}. 
Given $(f,h)\in\mathcal{O},$ it is easy to see that  
\begin{align}
 &\mathcal{A}(f)=\partial_{xx}-2\frac{d+y}{d_1+f}f^{\prime}\partial_{xy}+\frac{|(d+y)f'|^2+d_1^2}{(d_1+f)^2}\partial_{yy}-(d+y)\frac{(d_1+f)f^{\prime\prime}-2f^{\prime 2}}
 {(d_1+f)^2}\partial_{y},\label{Af}\\
& \mathcal{A}(f,h)=\partial_{xx}-2\frac{d_2h^{\prime}+(h-f)^{\prime}y}{h-f+d_2}\partial_{xy}+\frac{|d_2h^{\prime}+(h-f)^{\prime}y|^2+d_2^2}{(h-f+d_2)^2}\partial_{yy}\nonumber\\
 &\hspace{1.6cm}-\Big[\frac{d_2h^{\prime\prime}+(h-f)^{\prime\prime}y}{h-f+d_2}-2\frac{d_2h^{\prime}(h-f)^{\prime}+(h^{\prime}-f^{\prime})^2y}{(h-f+d_2)^2}\Big]\partial_{y},\label{Afh}
\end{align}
respectively, given $(w_1,w_2)\in C^{3+\alpha}_{e,per}(\overline\Omega_1)\times C^{3+\alpha}_{e,per}(\overline\Omega_2)$ and $\Lambda\in\mathbb{R}$, we have that 
\begin{align}
 &\mathcal{B}_1(\Lambda,(f,h),w_2)=\Big[|\partial_xw_2|^2-\frac{2d_2h^{\prime}}{h-f+d_2}\partial_x w_2\partial_y w_2 +\frac{d_2^2(1+h'^2)}{(h-f+d_2)^2}|\partial_y w_2|^2\Big]\Big|_{y=0}\nonumber\\
 &\hspace{3.1cm}+2g(d+h)-Q(\Lambda),\label{B1}\\
 &\mathcal{B}_2(\Lambda,(f,h))[w_1,w_2]=\frac{d_2}{h-f+d_2}\partial_{y}w_2\Big|_{y=-d_2}-\frac{d_1}{d_1+f}\partial_{y}w_1\Big|_{y=-d_2}.\label{B2}
\end{align}
\begin{proof}[Proof of Lemma \ref{L:FD}]
 Since
\begin{equation*}
\partial_{(f,h)}\Phi(\Lambda,0)[(f,h)]=\begin{pmatrix}\partial_f\Phi_1 (\Lambda,0)[f] & \partial_h\Phi_1 (\Lambda,0)[h]\\[1ex]
                                           \partial_f\Phi_2 (\Lambda,0)[f] & \partial_h\Phi_2 (\Lambda,0)[h] 
                                        \end{pmatrix}
\end{equation*}
we only need to determine the entries in the matrix $\partial_{(f,h)}\Phi(\Lambda,0).$\medskip

\noindent{\em The derivative $\partial_f\Phi_1 (\Lambda,0)$}
Using the definition of $\Phi_1$, we see that  
\begin{equation}\label{phi1f_i}
\partial_f\Phi_1 (\Lambda,0)[f]=2\Big[\frac{f}{d_2}|\partial_y \psi^{0}_{2}|^2+ \partial_y \psi^{0}_{2} \partial_y(\partial_f w_2(\Lambda,0)[f])\Big]\Big|_{y=0},
\end{equation}
whereby $\partial_y \psi^{0}_{2}\big|_{y=0}=\Lambda$ and $z:=\partial_f w_2(\Lambda,0)[f]$ is, in view of \eqref{DP2}, the solution of the Dirichlet problem
\begin{equation}\label{z1}
\left\{\begin{array}{lll}
 \Delta z=-\partial_f\mathcal{A}(0,0)[f]\psi_2^0 &\text{ in $\Omega_2,$}\\[0.3ex]
 z=0&\text{on $ \partial\Omega_2$}.
\end{array}\right.
\end{equation}
 A routine calculation shows now that 
$$\partial_f\mathcal{A}(0,0)[f]\psi_2^0=\frac{2\gamma_2 f}{d_2}+\Big(\frac{\gamma_2}{d_2}y^2+ \frac{\Lambda}{d_2} y\Big)f^{\prime\prime}.$$
Expanding $f$ and $z(y), y\in[-d_2,0],$  by their Fourier series, we have 
$$f=\sum_{k\in\mathbb{N}}a_k\cos\left(R_kx\right)\qquad{\rm and}\qquad z(y)=\sum_{k\in\mathbb{N}}a_kA_k(y)\cos\left(R_kx\right).$$
 The coefficients $A_k$ solve, in view of  \eqref{z1}, the following boundary value problem
$$
\left\{\begin{array}{lll}
        A_k''-R_k^2A_k=-\frac{2\gamma_2 }{d_2}+R_k^2
        \Big(\frac{\gamma_2}{d_2}y^2+\frac{\Lambda}{d_2}y\Big)& \text{in $(-d_2,0)$},\\[0.3ex]
        A_k(-d_2)=A_k(0)=0,
       \end{array}\right.
$$
and therefore
$$
A_k(y)=(\Lambda-\gamma_2 d_2)\frac{\sinh(R_ky)}{\sinh(R_kd_2)}- 
 \Big(\frac{\gamma_2}{d_2}y^2+\frac{\Lambda}{d_2}y\Big).$$
 Using the relation \eqref{phi1f_i} we obtain now that  
\begin{equation*}
\partial_f\Phi_1 (\Lambda,0)[f]=\sum_{k\in\mathbb{N}}m^{11}_k a_k\cos\left(R_kx\right),
 \end{equation*}
whereby $(m^{11}_k)_{k\in\mathbb{N}}$ is defined by \eqref{a11}. 
\medskip

\noindent{\em The derivative $\partial_h\Phi_1 (\Lambda,0)$}
We have that
$$\partial_h\Phi_1(\Lambda,0)[h]=2\Big[-\frac{h}{d_2}|\partial_y\psi^0_2|^2+\partial_y\psi^0_2\partial_y (\partial_h w_2(\Lambda,0)[h])\Big]\Big|_{y=0}+2gh,$$
with $z:=\partial_h w_2(\Lambda,0)[h])$ solving the Dirichlet problem
$$
\left\{\begin{array}{lll}
 \Delta z=-\partial_h\mathcal{A}(0,0)[h]\psi_2^0 &\text{in $\Omega_2,$}\\[0.3ex]
 z=0&\text{on $ \partial\Omega_2$},
\end{array}\right.
$$
cf. \eqref{DP2}.
Recalling \eqref{Afh}, we compute that
$$\partial_h\mathcal{A}(0,0)[h]\psi_2^0 =-\frac{2\gamma_2}{d_2}h-\Big(\frac{\gamma_2}{d_2}y^2+\frac{\gamma_2d_2+\Lambda}{d_2}y+\Lambda\Big)h^{\prime\prime}. $$
Using Fourier expansions as before, that is   
$$h=\sum_{k\in\mathbb{N}}b_k\cos\left(R_kx\right)\qquad{\rm and}\qquad z(y)=\sum_{k\in\mathbb{N}}b_kB_k(y)\cos\left(R_kx\right)\quad\text{for $ y\in[-d_2,0],$}$$
 we obtain that the coefficients $B_k$ satisfy
$$
\left\{\begin{array}{lll}
        B_k''-R_k^2B_k=\frac{2\gamma_2 }{d_2}-R_k^2
        \Big(\frac{\gamma_2}{d_2}y^2+\frac{\gamma_2d_2+\Lambda}{d_2}y+\Lambda\Big)&\text{in $(-d_2,0)$},\\[0.3ex]
        B_k(-d_2)=B_k(0)=0.
       \end{array}\right.
$$
The solution of this boundary value problem is
\begin{align*}
B_k(y)=&-\Lambda\Big(\frac{\sinh\left(R_ky\right)}{\tanh\left(R_kd_2\right)} +\cosh\left(R_ky\right)\Big)+\frac{\gamma_2}{d_2}y^2+\frac{\gamma_2d_2+\Lambda}{d_2}y+\Lambda,\label{Bk}
\end{align*}
and the desired representation for the derivative $\partial_h\Phi_1 (\Lambda,0)$ follows at once.
\medskip

\noindent{\em The derivative $\partial_f\Phi_2 (\Lambda,0)$}
From the definition of $\Phi_2$ we obtain that 
\begin{align*}
 \partial_f\Phi_2(\Lambda,0)[f]= \Big[\frac{f}{d_2}\partial_y\psi^0_2+\frac{f}{d_1}\partial_y\psi^0_1+\partial_y\big(\partial_f w_2(\Lambda,0)[f] - \partial_f w_1(\Lambda,0)[f]\big)\Big]\Big|_{y=-d_2},
\end{align*}
whereby, in the equality above, $z:=\partial_f w_1(\Lambda,0)[f]$ solves the Dirichlet problem
$$
\left\{\begin{array}{lll}
 \Delta z=-\partial_f\mathcal{A}(0)[f]\psi_1^0 &\text{ in $\Omega_1$},\\
 z=0&\text{ on  $\partial\Omega_1.$}
\end{array}\right.
$$
In view of \eqref{Af}, we compute that 
\begin{align*}
\partial_f\mathcal{A}(0)[f]\psi_1^0=&-\frac{2\gamma_1 f}{d_1}- \Big[\frac{\gamma_1}{d_1}y^2
+\Big(\frac{d+d_2}{d_1}\gamma_1 +\frac{\Lambda-\gamma_2d_2}{d_1}\Big)y +\frac{d\Lambda}{d_1}+\frac{d d_2(\gamma_1-\gamma_2)}{d_1}\Big] f''.
\end{align*}
Expanding $f$ and $z(y), y\in[-d,-d_2],$ by their  Fourier series 
$$f=\sum_{k\in\mathbb{N}}a_k\cos\left(R_kx\right)\qquad{\rm and}\qquad z(y)=\sum_{k\in\mathbb{N}}a_kC_k(y)\cos\left(R_kx\right),$$
we find  that the coefficient  $C_k$ is the solution of 
$$
\left\{\begin{array}{llll}
        C_k^{\prime\prime}-R_k^2 C_k=\frac{2\gamma_1 }{d_1}-R_k^2 \Big[\frac{\gamma_1}{d_1}y^2
+\Big(\frac{d+d_2}{d_1}\gamma_1 +\frac{\Lambda-\gamma_2d_2}{d_1}\Big)y +\frac{d\Lambda}{d_1}+\frac{d d_2(\gamma_1-\gamma_2)}{d_1}\Big]&\text{in $(-d,-d_2),$}\\[0.3ex]
        C_k(-d)=C_k(-d_2)=0,
       \end{array}\right.
$$
whence
\begin{align*}
 C_k(y)=&\Lambda\frac{\sinh\left((d+y)R_k\right)}{\sinh\left(R_k d_1\right)}+\frac{\gamma_1}{d_1}y^2
+\Big(\frac{d+d_2}{d_1}\gamma_1 +\frac{\Lambda-\gamma_2d_2}{d_1}\Big)y +\frac{d\Lambda}{d_1}+\frac{d d_2(\gamma_1-\gamma_2)}{d_1}.
\end{align*}
The representation of $\partial_f\Phi_2 (\Lambda,0)$ as a Fourier multiplier follows now easily. 
\medskip

\noindent{\em The derivative $\partial_h\Phi_2 (\Lambda,0)$}
Observing that  
\begin{align*}
 \partial_h\Phi_2(\Lambda,0)[h]=\frac{\gamma_2d_2-\Lambda}{d_2}h +\partial_y(\partial_h w_2(\Lambda,0)[h])|_{y=-d_2},
\end{align*}
 the desired representation for $\partial_h\Phi_2 (\Lambda,0)$ follows by using the expression  for $\partial_h w_2(\Lambda,0)[h]$ determined in the second part of this proof.
\end{proof}

In the remaining part we establish the  Lemmas \ref{fig1_left}-\ref{fig3_right} that   provide the justification for the streamlines pattern, as seen from a reference frame moving with the wave, as shown in Figures \ref{Fig1}-\ref{Fig3}.
Because the proofs of Lemmas \ref{fig1_left}-\ref{fig3_right} use similar arguments, we present herein only the proof  for Lemma \ref{fig1_left}. 
For this, it is important to note that because there is no time dependence in problem \eqref{VF} (or \eqref{PB2}), the particle trajectories and the streamlines corresponding to the solutions found in Theorems \ref{MT1}-\ref{MT3}, \ref{MT4}, \ref{MT5}  
coincide with the level curves of  the corresponding stream function.
They are parametrized by   solutions of the system of ordinary differential equations
 \begin{equation}\label{Eqq}
  \left\{\begin{array}{l}
  x'=u-c=\psi_{y}, \\
 y'=v=-\psi_{x},
 \end{array}\right.
 \end{equation}
 stagnation points of the flows corresponding to equilibria of \eqref{Eqq}.
 Hence, our task is to determine the level curves of the stream function.
 The direction of motion of the particles along the level curves is determined by the sign of $u-c$ or $v$.

\begin{lem}\label{fig1_left}
 Assume that $\gamma_1>\gamma_2>0$ and  let
 \[
((f,h),\psi_1,\psi_2)\in \big(C^{3+\alpha}_{per}(\mathbb{R})\big)^2\times C^{3+\alpha}_{per}\big(\, \overline{\Omega(f)}\, \big)\times C^{3+\alpha}_{per}\big(\, \overline{\Omega(f,h)}\, \big) 
\]
be a solution of \eqref{PB2} that is determined by a point $(\Lambda(s), (f,h)(s))$ on the bifurcation curve  found in Theorem \ref{MT1}.
Provided that $s$ is small enough, the following assertions are true:
 \begin{enumerate}
  \item [ $(i)$ ] $f'>0$ and $h'>0$ on $\left(0,L/2\right);$
  \item [$(ii)$ ] $\partial_x\psi_{2}<0$ in $\{(x,y)\in\Omega(f,h)\,:\, x\in(0,L/2)\}$ and $\partial_y\psi_{2} >0$ in  $\Omega(f,h);$ 
  \item [$(iii)$ ] $\partial_x\psi_{1}<0$  in $\{(x,y)\in\Omega(f)\,:\, x\in(0,L/2)\};$
 \item [$(iv)$ ] There is a  smooth curve $\{(x,\xi(x)):x\in [0,L/2]\}$ with $-d<\xi(x)<-d_2+f(x)$ for all $x\in[0,L/2]$ and satisfying additionally:
 \begin{enumerate}
 \item [ $(a)$ ] Given $x\in [0,L/2]$, it holds that: $\partial_y\psi_{1}(x,\xi (x))=0$, $\partial_y\psi_{1}(x,y)<0$ for all $y\in[-d,\xi(x)), $ and $\partial_y\psi_{1}(x,y)>0$ for all $y\in(\xi(x), -d_2+f(x)];$
 \item [ $(b)$ ] $\xi$ is strictly decreasing on $[0,L/2]$;
 \item [ $(c)$ ] The function $[x\mapsto \psi_1(x,\xi (x))]$ is strictly decreasing on $[0,L/2]$.
 \end{enumerate} 
 \end{enumerate}
 \end{lem}
\begin{proof}
Since $\Lambda(0)=\Lambda_1\in(\gamma_2 d_2,\gamma_1 d_1+\gamma_2 d_2)$, for small $s$  it holds $\Lambda(s)\in(\gamma_2 d_2,\gamma_1 d_1+\gamma_2 d_2)$.
Recalling that 
\begin{equation}
f(s)=s m^{22}_{1}(\Lambda_1)\cos\left(\frac{2\pi}{L}x\right)+O(s^2),\quad h(s)=-s m^{21}_{1}(\Lambda_1)\cos\left(\frac{2\pi}{L}x\right)+O(s^2),
\end{equation}
with $m^{22}_{1}(\Lambda_1)<0$ and $m^{21}_{1}(\Lambda_1)>0$, the claim $(i)$ follows  by using the same arguments as in the proof of  \cite[Lemma 4.2]{W09}. 

For $(ii),$ we see first that $\partial_y\psi_{2}^{0}=\gamma_2 y+\Lambda_1>\gamma_2 y+\gamma_2 d_2>0$ in $\overline\Omega_2$. 
Therefore, $\partial_y\psi_{2} >0$ in  $\Omega(f,h)$ provided that $s$ is small.
Using now $(i)$ and the fact that $\psi$ is constant on $\partial\Omega(f,h)$ and even with respect to $x$, it is easy to see that $\partial_x\psi_{2}\leq 0$ on the boundary of the set $\{(x,y)\in\Omega(f,h)\,:\, x\in(0,L/2)\}.$ 
Observing that $\partial_x\psi_{2}(x,h(x))<0$ for all $x\in(0,L/2)$ and $\Delta\psi_x=0$ in $\Omega(f,h),$ elliptic maximum principles ensure that $\partial_x\psi_2<0$ in $\{(x,y)\in\Omega(f,h)\,:\, x\in(0,L/2)\}.$
The claim $(iii)$ is obtained in a similar manner.

For $(iv),$ we remark that 
\begin{equation*}
\partial_y\psi^{0}_{1}\big|_{y=-d}<0,\qquad \partial_y\psi^{0}_{1}\big|_{y=-d_2}>0,\qquad \partial_{yy}\psi^{0}_1>0 \text{ in $\overline\Omega_1.$}
\end{equation*}
Therewith, for small $s$ the function $\psi_1$ satisfies the similar inequalities
\begin{equation}\label{cont_cond}
\partial_y\psi_{1}\big|_{y=-d}<0,\qquad \partial_y\psi_{1}\big|_{y=-d_2+f}>0,\qquad \partial_{yy}\psi_1>0 \text{ in $\overline\Omega(f).$}
\end{equation}
Hence, for each $x\in [0,L/2]$, there exists a unique $\xi(x)\in(-d,-d_2+f(x))$ such that $\partial_y\psi_1(x,\xi(x))=0$.
Due to the third inequality in \eqref{cont_cond} we conclude
from the implicit function theorem that $\xi$ is smooth and 
\begin{equation}\label{TL}
 \partial_{xy}\psi_{1}(x,\xi(x))+\xi'(x)\partial_{yy}\psi_{1}(x,\xi(x))=0\qquad\text{for all $x\in[0,L/2].$}
\end{equation}
We are going to determine now the sign of $\partial_{xy}\psi_{1}$. 
To this end note that $\psi_1 =w_1\circ \Phi_{f}^{-1}$ where $w_1\in C^{3+\alpha}_{e,per}(\overline{\Omega}_1)$ is the unique solution
of the problem \eqref{DP1}, that is $w_1:=w_1(\Lambda(s), (f,h)(s)).$ 
By the chain rule we get 
\begin{align*}
 \partial_{xy}\psi_1=&-\frac{d_1 f^{\prime} }{(d_1 +f)^2}\partial_y w_1\circ \Phi_{f}^{-1} +\frac{d_1}{d_1+f}\partial_{xy}w_1\circ \Phi_{f}^{-1}-\frac{d_1^2 f^{\prime} y}{(d_1+f)^3}\partial_{yy}w_1\circ \Phi_{f}^{-1} \\
 &+\frac{d d_1 ff^{\prime}}{(d_1+f)^3}\partial_{yy}w_1\circ \Phi_{f}^{-1}-\frac{d d_1 f^{\prime} }{(d_1 +f)^2}\partial_{yy}w_1\circ \Phi_{f}^{-1}.
\end{align*}
On the other hand  we have the following expansion
\begin{equation*}
 w_1(\Lambda(s) ,(f, h)s))=w_1(\Lambda_1,0)+\partial_{\Lambda}w_1(\Lambda_1, 0)[\Lambda(s)-\Lambda_1]+\partial_f w_1(\Lambda_1 ,0)[f(s)]+O(s^2) 
\end{equation*}
in $C^{3+\alpha}_{e,per}(\overline\Omega_1).$
Observing that
\begin{align*}
\left.\begin{array}{lll}
 &\partial_y w_1\circ \Phi_{f}^{-1}=\partial_y\psi^{0}_1+O(s),\\
 &\partial_{xy}w_1\circ \Phi_{f}^{-1}=\partial_{xy}(\partial_f w_1(\Lambda_1 ,0)[f])+O(s^2),\\
 &\partial_{yy}w_1\circ \Phi_{f}^{-1}=\gamma_1 +O(s),
\end{array}\right\}\qquad\text{in $C^{2+\alpha}_{e,per}\big(\, \overline{\Omega(f)}\, \big)$,}
\end{align*}
a lengthy calculation leads us to 
$$\partial_{xy}\psi_1=-s m^{22}_{1}(\Lambda_1)L_1\Lambda_1 \frac{\cosh(L_1(d+y))}{\sinh(L_1 d_1)}\sin\left(L_1 x\right)+O(s^2)\qquad\text{in $C^{1+\alpha}_{e,per}\big(\, \overline{\Omega(f)}\, \big)$}.$$
A similar argument to the one used in $(i)$ shows that $\partial_{xy}\psi_1>0$ in  $\Omega(f)$ if $s>0$ is sufficiently small.
The latter property together with \eqref{cont_cond} and \eqref{TL} implies that 
$\xi' <0$ in $x\in\left(0,L/2\right)$.
This proves the claim in $(b)$. 
Since $(c)$ is  an obvious consequence of $(iii)$ we have completed the proof.
\end{proof}

It follows now readily from Theorem \ref{MT1} and Lemma \ref{fig1_left} that the streamline pattern in the moving frame for the non laminar solutions found in 
Theorem \ref{MT1} for $\gamma_1>\gamma_2$ is as in Figure \ref{Fig1} (left picture).
The next lemma justifies the right picture of Figure \ref{Fig1}.

\begin{lem}\label{fig1_right}
 Assume that $\gamma_1<\gamma_2,\gamma_2>0$ and let
 \[
((f,h),\psi_1,\psi_2)\in \big(C^{3+\alpha}_{per}(\mathbb{R})\big)^2\times C^{3+\alpha}_{per}\big(\, \overline{\Omega(f)}\, \big)\times C^{3+\alpha}_{per}\big(\, \overline{\Omega(f,h)}\, \big) 
\]
be a solution of \eqref{PB2} that is determined by a point $(\Lambda(s), (f,h)(s))$ on the bifurcation curve  found in Theorem \ref{MT1}.
Provided that $s$ is small enough, the following assertions are true:
 \begin{enumerate}
  \item [ $(i)$ ] $f'>0$ and $h'<0$ on $\left(0,L/2\right);$
  \item [$(ii)$ ] $\partial_x\psi_{1}>0$  in $\{(x,y)\in\Omega(f)\,:\, x\in(0,L/2)\}$ and $\partial_y\psi_{1} <0$ in  $\Omega(f);$ 
  \item [$(iii)$ ] $\partial_x\psi_2>0$  in  in $\{(x,y)\in\Omega(f,h)\,:\, x\in(0,L/2)\};$
 \item [$(iv)$ ]There is a  smooth curve $\{(x,\xi(x)):x\in [0,L/2]\}$ with $-d_2+f(x)<\xi(x)<h(x)$ for all $x\in[0,L/2]$ and satisfying additionally:
 \begin{enumerate}
 \item [ $(a)$ ] Given $x\in [0,L/2]$, it holds that: $\partial_y\psi_{2}(x,\xi (x))=0$, $\partial_y\psi_{2}(x,y)<0$ for all $y\in[-d_2+f(x),\xi(x)), $ and $\partial_y\psi_{2}(x,y)>0$ for all $y\in(\xi(x), h(x)];$
 \item [ $(b)$ ] $\xi$ is strictly decreasing on $[0,L/2]$;
 \item [ $(c)$ ] The function $[x\mapsto \psi_2(x,\xi (x))]$ is strictly increasing on $[0,L/2]$.
 \end{enumerate} 
 \end{enumerate}
 \end{lem}

 The next lemma provides a justification for the left picture of Figure \ref{Fig2}.
 
\begin{lem}\label{fig2_left}
Assume that $\gamma_2>0, \gamma_1 d_1+\gamma_2 d_2>0$ and let 
\[
((f,h),\psi_1,\psi_2)\in \big(C^{3+\alpha}_{per}(\mathbb{R})\big)^2\times C^{3+\alpha}_{per}\big(\, \overline{\Omega(f)}\, \big)\times C^{3+\alpha}_{per}\big(\, \overline{\Omega(f,h)}\, \big) 
\]
be a solution of \eqref{PB2} that is determined by a point $(\Lambda(s), (f,h)(s))$ on the bifurcation curve  found in Theorem \ref{MT2}.
Provided that $s$ is small enough, the following assertions are true:
 \begin{enumerate}
  \item [ $(i)$ ] $f'>0$ and $h'<0$ on $\left(0,L/2\right);$
  \item [$(ii)$ ] $\partial_x\psi_{1}>0$  in $\{(x,y)\in\Omega(f)\,:\, x\in(0,L/2)\}$ and $\partial_y\psi_{1} <0$ in  $\Omega(f);$ 
  \item [$(iii)$ ] $\partial_x\psi_2>0$  in  in $\{(x,y)\in\Omega(f,h)\,:\, x\in(0,L/2)\};$
 \item [$(iv)$ ] There is a smooth curve $\{(x,\xi(x)):x\in [0,L/2]\}$ with $-d_2+f(x)<\xi(x)<h(x)$ for all $x\in[0,L/2]$ and satisfying additionally:
 \begin{enumerate}
 \item [ $(a)$ ] Given $x\in [0,L/2]$, it holds that: $\partial_y\psi_{2}(x,\xi (x))=0$, $\partial_y\psi_{2}(x,y)<0$ for all $y\in[-d_2+f(x),\xi(x)), $ and $\partial_y\psi_{2}(x,y)>0$ for all $y\in(\xi(x), h(x)];$
 \item [ $(b)$ ] $\xi$ is strictly decreasing on $[0,L/2]$;
 \item [ $(c)$ ] The function $[x\mapsto \psi_2(x,\xi (x))]$ is strictly increasing on $[0,L/2]$.
 \end{enumerate} 
 \end{enumerate}
\end{lem}

We provide now a justification for the right picture of Figure \ref{Fig2}.

\begin{lem}\label{fig2_right}
Assume that $\gamma_2>0, \gamma_1 d_1+\gamma_2 d_2\leq 0$ and let 
\[
((f,h),\psi_1,\psi_2)\in \big(C^{3+\alpha}_{per}(\mathbb{R})\big)^2\times C^{3+\alpha}_{per}\big(\, \overline{\Omega(f)}\, \big)\times C^{3+\alpha}_{per}\big(\, \overline{\Omega(f,h)}\, \big) 
\]
be a solution of \eqref{PB2} that is determined by a point $(\Lambda(s), (f,h)(s))$ on the bifurcation curve  found in Theorem \ref{MT2}.
Provided that $s$ is small enough, the following assertions are true:
 \begin{enumerate}
  \item [ $(i)$ ] $f'>0$ and $h'<0$ on $\left(0,L/2\right);$
  \item [$(ii)$ ] $\partial_x\psi_{1}>0$  in $\{(x,y)\in\Omega(f)\,:\, x\in(0,L/2)\}$; 
  \item [$(iii)$ ] $\partial_x\psi_2>0$  in  in $\{(x,y)\in\Omega(f,h)\,:\, x\in(0,L/2)\};$
   \item [$(iv)$ ] There is a smooth curve $\{(x,\xi_1(x)):x\in [0,L/2]\}$ with $-d<\xi_1(x)<-d_2+f(x)$ for all $x\in[0,L/2]$ and satisfying additionally:
 \begin{enumerate}
 \item [ $(a)$ ] Given $x\in [0,L/2]$, it holds that: $\partial_y\psi_{1}(x,\xi_1(x))=0$, $\partial_y\psi_{1}(x,y)>0$ for all $y\in[-d,\xi_1(x)), $ and $\partial_y\psi_{1}(x,y)<0$ for all $y\in(\xi_1(x), -d_2+f(x)];$
 \item [ $(b)$ ] $\xi_1$ is strictly increasing on $[0,L/2]$;
 \item [ $(c)$ ] The function $[x\mapsto \psi_1(x,\xi_1 (x))]$ is strictly increasing on $[0,L/2]$.
 \end{enumerate} 
 \item [$(v)$ ] There is a smooth curve $\{(x,\xi_2(x)):x\in [0,L/2]\}$ with $-d_2+f(x)<\xi_2(x)<h(x)$ for all $x\in[0,L/2]$ and satisfying additionally:
 \begin{enumerate}
 \item [ $(a)$ ] Given $x\in [0,L/2]$, it holds that: $\partial_y\psi_{2}(x,\xi_2(x))=0$, $\partial_y\psi_{2}(x,y)<0$ for all $y\in[-d_2+f(x),\xi_2(x)), $ and $\partial_y\psi_{2}(x,y)>0$ for all $y\in(\xi_2(x), h(x)];$
 \item [ $(b)$ ] $\xi_2$ is strictly decreasing on $[0,L/2]$;
 \item [ $(c)$ ] The function $[x\mapsto \psi_2(x,\xi_2 (x))]$ is strictly increasing on $[0,L/2]$.
 \end{enumerate} 
 \end{enumerate}
 \end{lem}
 
 We consider now the non laminar flows corresponding to the  bifurcation solutions found in Theorem \ref{MT3} and prove the following result which justifies the left picture of Figure \ref{Fig3}.
 
\begin{lem}\label{fig3_left}
Assume that $\gamma_2>0, \gamma_1 d_1+\gamma_2 d_2<0$ and let 
\[
((f,h),\psi_1,\psi_2)\in \big(C^{3+\alpha}_{per}(\mathbb{R})\big)^2\times C^{3+\alpha}_{per}\big(\, \overline{\Omega(f)}\, \big)\times C^{3+\alpha}_{per}\big(\, \overline{\Omega(f,h)}\, \big) 
\]
be a solution of \eqref{PB2} that is determined by a point $(\Lambda(s), (f,h)(s))$ on the bifurcation curve  found in Theorem \ref{MT3}.
Provided that $s$ is small enough, the following assertions are true:
 \begin{enumerate}
  \item [ $(i)$ ] $f'<0$ and $h'<0$ on $\left(0,L/2\right);$
  \item [$(ii)$ ] $\partial_x\psi_{2}<0$  in $\{(x,y)\in\Omega(f,h)\,:\, x\in(0,L/2)\}$ and $\partial_y\psi_{2} <0$ in  $\Omega(f,h);$ 
  \item [$(iii)$ ] $\partial_x\psi_1<0$  in   $\{(x,y)\in\Omega(f)\,:\, x\in(0,L/2)\};$
 \item [$(iv)$ ] There is a smooth curve $\{(x,\xi(x)):x\in [0,L/2]\}$ with $-d<\xi(x)<-d_2+f(x)$ for all $x\in[0,L/2]$ and satisfying additionally:
 \begin{enumerate}
 \item [ $(a)$ ] Given $x\in [0,L/2]$, it holds that: $\partial_y\psi_{1}(x,\xi (x))=0$, $\partial_y\psi_{1}(x,y)>0$ for all $y\in[-d,\xi(x)), $ and $\partial_y\psi_{1}(x,y)<0$ for all $y\in(\xi(x), -d_2+f(x)];$
 \item [ $(b)$ ] $\xi$ is strictly increasing on $[0,L/2]$;
 \item [ $(c)$ ] The function $[x\mapsto \psi_1(x,\xi (x))]$ is strictly decreasing on $[0,L/2]$.
 \end{enumerate} 
  \end{enumerate} 
\end{lem}
Finally, we have the following result which justifies the right picture of Figure \ref{Fig3}.

\begin{lem}\label{fig3_right}
Assume that $\gamma_2=0$, $\gamma_1<0,$ and let 
\[
((f,h),\psi_1,\psi_2)\in \big(C^{3+\alpha}_{per}(\mathbb{R})\big)^2\times C^{3+\alpha}_{per}\big(\, \overline{\Omega(f)}\, \big)\times C^{3+\alpha}_{per}\big(\, \overline{\Omega(f,h)}\, \big) 
\]
be a solution of \eqref{PB2} that is determined by a point $(\Lambda(s), (f,h)(s))$ on one of  the bifurcation curves  found in Theorems \ref{MT4}-\ref{MT5}.
Then, the assertions from Lemma \ref{fig3_left} hold verbatim.
\end{lem}

\vspace{0.5cm}
\noindent{\bf Acknowledgements} 
The authors thank the anonymous referees for  the comments and suggestions which have improved the quality of the article.


\begin{thebibliography}{10}

\bibitem{B62}
B.~T. Benjamin.
\newblock {\em  The solitary wave on a stream with an arbitrary distribution of
  vorticity}.
\newblock { J. Fluid Mech.}, 12, 97--116, 1962.

\bibitem{CoCla13}
D.~Clamond and A.~Constantin.
\newblock {\em  Recovery of steady periodic wave profiles from pressure measurements
  at the bed}.
\newblock { J. Fluid Mech.}, 714, 463--475, 2013.

\bibitem{Co06}
A.~Constantin.
\newblock {\em  The trajectories of particles in {S}tokes waves}.
\newblock {  Invent. Math.}, 166(3), 523--535, 2006.

\bibitem{Cobook}
A.~Constantin.
\newblock {\em {Nonlinear water waves with applications to wave-current
  interactions and tsunamis}}, volume~81 of {\em {CBMS-NSF Regional Conference
  Series in Applied Mathematics}}.
\newblock Society for Industrial and Applied Mathematics (SIAM), Philadelphia,
  PA, 2011.

\bibitem{AC12a}
A.~Constantin.
\newblock {\em  Dispersion relations for periodic traveling water waves in flows
  with discontinuous vorticity}.
\newblock {  Commun. Pure Appl. Anal.}, 11(4), 1397--1406, 2012.

\bibitem{CoARMA13}
A.~Constantin.
\newblock {\em  Mean velocities in a {S}tokes wave}.
\newblock {  Arch. Ration. Mech. Anal.}, 207(3), 907--917, 2013.

\bibitem{CoSt04}
A.~Constantin and W.~Strauss.
\newblock {\em  Exact steady periodic water waves with vorticity}.
\newblock {  Comm. Pure Appl. Math.}, 57(4), 481--527, 2004.

\bibitem{CoSt10}
A.~Constantin and W.~Strauss.
\newblock {\em  Pressure beneath a {S}tokes wave}.
\newblock {  Comm. Pure Appl. Math.}, 63(4), 533--557, 2010.

\bibitem{CoSt11}
A.~Constantin and W.~Strauss.
\newblock {\em  Periodic traveling gravity water waves with discontinuous
  vorticity}.
\newblock {  Arch. Ration. Mech. Anal.}, 202(1), 133--175, 2011.

\bibitem{CV11}
A.~Constantin and E.~Varvaruca.
\newblock {\em  Steady periodic water waves with constant vorticity: regularity and
  local bifurcation}.
\newblock {  Arch. Ration. Mech. Anal.}, 199(1), 33--67, 2011.

\bibitem{CR71}
M.~G. Crandall and P.~H. Rabinowitz.
\newblock {\em  Bifurcation from simple eigenvalues}.
\newblock {  J. Functional Analysis}, 8, 321--340, 1971.

\bibitem{EEW11}
M.~Ehrnstr\"{o}m, J.~Escher, and E.~Wahl\'{e}n.
\newblock {\em  Steady water waves with multiple critical layers}.
\newblock {  SIAM J. Math. Anal.}, 43(3), 1436--1456, 2011.

\bibitem{EW14x}
M.~Ehrnstr\"{o}m and E.~Wahl\'{e}n.
\newblock {\em  Trimodal steady water waves}.
\newblock {  Arch. Rational Mech. Anal.}, 216(2), 449–471, 2015.
 

\bibitem{EMM11}
J.~Escher, A.-V. Matioc, and B.-V. Matioc.
\newblock {\em  On stratified steady periodic water waves with linear density
  distribution and stagnation points}.
\newblock {  J. Differential Equations}, 251(10), 2932--2949, 2011.

\bibitem{GT01}
D.~Gilbarg and N.~S. Trudinger.
\newblock {\em {Elliptic Partial Differential Equations of Second Order}}.
\newblock Springer Verlag, 2001.

\bibitem{Jon}
I.~G. Jonsson.
\newblock {\em {Wave-current interactions. In: The Sea}}.
\newblock Wiley, New York, 1990.

\bibitem{Ki65}
W.~Kinnersley.
\newblock {\em  Exact large amplitude capillary waves on sheets of fluid}.
\newblock {  J. Fluid Mech.}, 77(2), 229--241, 1976.

\bibitem{KO1}
J.~Ko and W.~Strauss.
\newblock {\em  Large-amplitude steady rotational water waves}.
\newblock {  Eur. J Mech. B Fluids}, 27, 96--109, 2007.

\bibitem{KO2}
J.~Ko and W.~Strauss.
\newblock {\em  Effect of vorticity on steady water waves}.
\newblock { J. Fluid Mech.}, 608, 197--215, 2008.

\bibitem{KK14}
V.~Kozlov and N.~Kuznetsov.
\newblock {\em  Dispersion equation for water waves with vorticity and {S}tokes
  waves on flows with counter-currents}.
\newblock {  Arch. Rational Mech. Anal.}, 214(3), 971--1018, 2014.

\bibitem{JL12}
J.~LeCrone.
\newblock {\em  Elliptic operators and maximal regularity on periodic
  little-{H}\"{o}lder spaces}.
\newblock { J. Evol. Equ.}, 12(2), 295--325, 2012.

\bibitem{CM12x}
C.~I. Martin.
\newblock {\em  Local bifurcation and regularity for steady periodic
  capillary--gravity water waves with constant vorticity}.
\newblock {  Nonlinear Anal. Real World Appl.}, (14), 131--149, 2013.

\bibitem{CM13}
C.~I. Martin.
\newblock {\em  Local bifurcation for steady periodic capillary water waves with
  constant vorticity}.
\newblock { J. Math. Fluid Mech.}, 15(1), 155--170, 2013.

\bibitem{CM13_2}
C.~I. Martin and B.-V. Matioc.
\newblock {\em  Existence of Wilton ripples for water waves with constant vorticity
  and capillary effects}.
\newblock {  SIAM J. Appl. Math.}, 73(4), 1582--1595, 2013.

\bibitem{CM14}
C.~I. Martin and B.-V. Matioc.
\newblock {\em  Existence of capillary-gravity water waves with piecewise constant
  vorticity}.
\newblock { J. Differential Equations}, 256(8), 3086--3114, 2014.

\bibitem{CM13xxx}
C.~I. Martin and B.-V. Matioc.
\newblock {\em  Steady periodic water waves with unbounded vorticity: equivalent
  formulations and existence results}.
\newblock {  J. Nonlinear Sci.}, 24, 633--659, 2014.

\bibitem{Ma12a}
A.-V. Matioc.
\newblock {\em  Steady internal water waves with a critical layer bounded by the
  wave surface.}
\newblock {  J. Nonlinear Math. Phys.}, 19(1), 1250008, 21 p., 2012.

\bibitem{M13x}
B.-V. Matioc.
\newblock {\em  Global bifurcation for water waves with capillary effects and
  constant vorticity}.
\newblock {  Monatsh. Math.}, 174(3), 459--475, 2014.

\bibitem{M13xx}
B.-V. Matioc and A.-V. Matioc.
\newblock {\em  Capillary-gravity water waves with discontinuous vorticity:
  existence and regularity results}.
\newblock {  Comm. Math. Phys.}, 330, 859--886, 2014.

\bibitem{O82}
K.~Okuda.
\newblock {\em  Internal flow structure of short wind waves}.
\newblock {  Journal of the Oceanographical Society of Japan}, 38, 28--42,
  1982.

\bibitem{PB74}
O.~M. Phillips and M.~L. Banner.
\newblock {\em  Wave breaking in the presence of wind drift and swell}.
\newblock {  J. Fluid Mech.}, 66, 625--640, 1974.

\bibitem{Thom}
G.~Thomas and G.~Klopman.
\newblock {\em {Wave-current interactions in the nearshore region}}.
\newblock WIT, Southampton, United Kingdom, 1997.

\bibitem{Ti01}
J.-P. Tignol.
\newblock {\em {Galois' theory of algebraic equations}}.
\newblock World Scientific Publishing Co., Inc., River Edge, NJ, 2001.

\bibitem{To96}
J.~F. Toland.
\newblock {\em  {S}tokes waves}.
\newblock { Topol. Methods Nonlinear Anal.}, 7(1), 1--48, 1996.

\bibitem{W09}
E.~Wahl\'{e}n.
\newblock {\em Steady water waves with a critical layer}.
\newblock {J. Differential Equations}, 246(6), 2468--2483, 2009.

\end{thebibliography}
\end{document}